\documentclass[11pt, a4paper,leqno]{amsart}
\usepackage{amsmath,amsthm,amscd,amssymb,amsfonts, amsbsy}
\usepackage{latexsym}
\usepackage{exscale}
\usepackage{txfonts}

\usepackage[colorlinks,citecolor=red,pagebackref,hypertexnames=false]{hyperref}
\usepackage{pdfsync}

\usepackage{color}

\calclayout
\allowdisplaybreaks

\theoremstyle{plain}
\newtheorem{theorem}[equation]{Theorem}
\newtheorem{lemma}[equation]{Lemma}
\newtheorem{corollary}[equation]{Corollary}

\newtheorem{claim}[equation]{Claim}

\theoremstyle{definition}
\newtheorem{definition}[equation]{Definition}

\theoremstyle{remark}
\newtheorem{remark}[equation]{Remark}

\numberwithin{equation}{section}

\newcommand{\eps}{\varepsilon}

\newcommand{\dist}{\operatorname{dist}}

\newcommand{\re}{\mathbb{R}}
\newcommand{\rn}{\mathbb{R}^n}

\newcommand{\ree}{\mathbb{R}^{n+1}}

\newcommand{\dd}{\mathbb{D}}

\newcommand{\C}{\mathcal{C}}

\newcommand{\om}{\Omega}

\newcommand{\F}{\mathcal{F}}

\newcommand{\M}{\mathcal{M}}

\newcommand{\W}{\mathcal{W}}

\newcommand{\B}{\mathcal{B}}

\newcommand{\sbf}{{\bf S}}

\newcommand{\G}{\mathcal{G}}

\newcommand{\mut}{\mathfrak{m}}

\newcommand{\pom}{\partial\Omega}

\newcommand{\hm}{\omega}

\newcommand{\sub}{{\bf U}^*}

\renewcommand{\emptyset}{\mbox{\textup{\O}}}

\DeclareMathOperator{\diam}{diam}

\DeclareMathOperator{\interior}{int}

\begin{document}
\allowdisplaybreaks

\title[Harmonic measure and quantitative connectivity. Part I]{
	Harmonic measure and quantitative connectivity: geometric characterization of the $L^p$-solvability of the Dirichlet problem. Part I}

\author{Steve Hofmann}

\address{Steve Hofmann
\\
Department of Mathematics
\\
University of Missouri
\\
Columbia, MO 65211, USA} \email{hofmanns@missouri.edu}

\author{Jos\'e Mar{\'\i}a Martell}

\address{Jos\'e Mar{\'\i}a Martell\\
Instituto de Ciencias Matem\'aticas CSIC-UAM-UC3M-UCM\\
Consejo Superior de Investigaciones Cient{\'\i}ficas\\
C/ Nicol\'as Cabrera, 13-15\\
E-28049 Madrid, Spain} \email{chema.martell@icmat.es}

\thanks{The first author is 
supported by NSF grant DMS-1664047.
The second author acknowledges financial support from the Spanish Ministry of Economy and Competitiveness, through the ``Severo Ochoa" Programme for Centres of Excellence in R\&D (SEV-2015- 0554). He also acknowledges that the research leading to these results has received funding from the European Research Council under the European Union's Seventh Framework Programme (FP7/2007-2013)/ ERC agreement no. 615112 HAPDEGMT. 
In addition, both authors were supported by NSF Grant DMS-1440140 while in residence at the MSRI in Berkeley, California, during Spring semester 2017.}

\date{\today}
\subjclass[2000]{31B05, 35J25, 42B25, 42B37}

\keywords{
Harmonic measure, Poisson kernel,
uniform rectifiability, weak local John condition, big pieces of chord-arc domains, Carleson measures.}

\begin{abstract}
Let $\Omega\subset \ree$ be an open set, not necessarily connected,
with an $n$-dimensional uniformly rectifiable boundary.
We show that $\pom$ may be approximated in a ``Big Pieces" sense
by boundaries of chord-arc subdomains of $\Omega$, and hence that 
harmonic measure for $\Omega$ is weak-$A_\infty$ with respect to surface measure
on $\pom$, provided that $\Omega$ 
satisfies a certain weak version of a local John condition.   Under the further assumption that
$\Omega$ satisfies an interior Corkscrew condition, and combined
with our previous work, and with recent work
of Azzam, Mourgoglou and Tolsa, this yields a geometric characterization of domains whose
harmonic measure is quantitatively absolutely continuous with respect to surface measure and hence a characterization of the fact that the associated $L^p$-Dirichlet problem is solvable for 
some finite $p$. 
\end{abstract}

\maketitle

\tableofcontents

\section{Introduction}\label{s1}

A classical result of F. and M. Riesz \cite{RR} states that for a simply 
connected domain $\Omega$ 
in the complex plane,
rectifiability of $\pom$ implies that harmonic measure for $\Omega$ is absolutely continuous with respect to
arclength measure on the boundary.  A quantitative version of this theorem was later
proved by Lavrentiev \cite{Lav}.
More generally, if only a portion of the boundary is rectifiable, Bishop and Jones
\cite{BiJo} have shown that harmonic measure
is absolutely continuous with respect to arclength on that portion.  They also present a counter-example
to show that the result of \cite{RR}
may fail in the absence of some connectivity hypothesis
(e.g., simple connectedness).

In dimensions greater than 2, a fundamental result of Dahlberg \cite{Dah} establishes a quantitative version
of absolute continuity, namely that 
harmonic measure belongs to the class $A_\infty$ in an appropriate local sense  
(see Definitions \ref{defAinfty} and \ref{deflocalAinfty} below),
with respect to surface measure
on the boundary of a Lipschitz domain.

The result of Dahlberg was extended to the class of Chord-arc domains (see Definition \ref{def1.cad})
by David and Jerison \cite{DJe}, and independently by Semmes \cite{Se}.    The Chord-arc hypothesis
was weakened to that of a two-sided Corkscrew condition (Definition \ref{def1.cork}) 
by Bennewitz and Lewis
\cite{BL}, who then drew the conclusion that harmonic measure is weak-$A_\infty$ 
(in an appropriate local sense, see Definitions \ref{defAinfty} and \ref{deflocalAinfty})
with respect to 
surface measure on the boundary; the latter condition is similar to the $A_\infty$ condition, but without
the doubling property, and is the best conclusion that can be obtained under the weakened
geometric conditions considered in \cite{BL}.  We note that weak-$A_\infty$ is still a quantitative,
scale invariant version of absolute continuity.

More recently, J. Azzam \cite{Az},
has given a geometric
characterization of the $A_\infty$ property of harmonic measure with respect to 
surface measure for domains with Ahlfors-David regular (ADR) boundary  (see Definition \ref{defadr}).  
This work is related to our own, so let us describe it in a bit more detail.  Specifically, Azzam shows that
for a domain $\Omega$ with ADR boundary, harmonic measure is in $A_\infty$ with respect
to surface measure, if and only if 1) $\pom$ is uniformly rectifiable (UR)\footnote{This 
is a quantitative, 
scale-invariant version of rectifiability, see Definition \ref{defur} and the 
ensuing comments.},
and 2) $\Omega$ is semi-uniform in the sense of Aikawa and Hirata \cite{AH}.  The semi-uniform condition
is a connectivity condition which states that for some uniform constant $M$,
every pair of points $X\in \Omega$ and $y\in \pom$ may 
be connected by a rectifiable curve $\gamma=\gamma(y,X)$, with
$\gamma \setminus \{y\} \subset \Omega$,
with length $\ell(\gamma) \leq M|X-y|$, 
and which satisfies the ``cigar path" condition
\begin{equation}\label{cigar}
\min\left\{\ell\big(\gamma(y,Z)\big),\ell\big(\gamma(Z,X)\big)\right\} \,\leq \, M\dist(Z,\pom)\,,\quad \forall \, Z\in \gamma\,.
\end{equation}
Semi-uniformity is a weak version of the well known uniform condition, whose definition is similar, except 
that it applies to all pairs of points $X,Y\in \Omega$.   For example, the unit disk centered at the origin, 
with the slit 
$-1/2\leq x\leq 1/2, y=0$
removed, is semi-uniform, but not uniform.
It was shown in \cite{AH} that for a domain satisfying a John condition and the
Capacity Density Condition (in particular, for a domain with an ADR boundary), 
semi-uniformity characterizes the
doubling property of harmonic measure.    
The method of \cite{Az} is, broadly speaking, related to that of
\cite{DJe}, and of \cite{BL}.  In \cite{DJe}, the authors show that a Chord-arc domain $\Omega$ 
may be approximated in a ``Big Pieces" sense
(see \cite{DJe} or \cite{BL} for a precise statement; also cf. 
Definition \ref{def1.john} below) by Lipschitz subdomains $\Omega'\subset \Omega$; 
this fact allows one to reduce 
matters to the result of Dahlberg via
the maximum principle (a method which, to the present authors' knowledge, first appears in \cite{JK} in the 
context of $BMO_1$ domains).   
The same strategy, i.e., Big Piece approximation by  Lipschitz subdomains, is employed in \cite{BL}.
Similarly, in \cite{Az}, matters are reduced to the result of \cite{DJe}, by showing that
for a domain $\Omega$ with an ADR boundary,  $\Omega$ is 
semi-uniform with a uniformly rectifiable boundary
if and only if it has ``Very Big Pieces" of Chord-arc subdomains (see \cite{Az} for a precise statement of the latter condition).
As mentioned above, the converse direction is also treated in \cite{Az}.  In that case, given an 
interior Corkscrew condition (which holds automatically in the presence of the doubling property of harmonic measure),
and provided that $\pom$ is ADR, the $A_\infty$ (or even weak-$A_\infty$) property of harmonic measure was 
already known to imply uniform rectifiability of the boundary \cite{HM-4} (although the published version appears
in \cite{HLMN}; see also \cite{MT} for an alternative proof, and a somewhat more general result); as in \cite{AH}, 
semi-uniformity
follows from the doubling property, although in \cite{Az}, the author manages to show this while dispensing with
the John domain background assumption (given a harmlessly strengthened version of the doubling property).

In light of the example of \cite{BiJo}, it has been an important open problem to determine the minimal 
connectivity assumption, which, in conjunction with uniform rectifiability of the boundary, yields quantitative
absolute continuity of harmonic measure with respect to surface measure.
We observe that in \cite{Az}, the 
connectivity condition (semi-uniformity), is tied to the doubling property 
of harmonic measure, and not to absolute continuity.
In the present work, we impose a significantly milder connectivity hypothesis than semi-uniformity, and we 
then show that $\pom$ may be approximated in a big pieces sense by boundaries
of chord-arc subdomains, and hence that harmonic measure $\hm$ satisfies a weak-$A_\infty$ condition with respect to surface measure $\sigma$
on the boundary, provided
that $\pom$ is uniformly rectifiable.  The weak-$A_\infty$ conclusion is best possible in this generality:  indeed,
the stronger conclusion that $\hm\in A_\infty(\sigma)$, which entails doubling of $\hm$, necessarily requires
semi-uniformity, as Azzam has shown.  One may then combine our results here with our previous work \cite{HM-4}, and with recent work of Azzam, Mourgoglou and Tolsa \cite{AMT}, to obtain a geometric characterization of quantitative absolute continuity of harmonic measure (see Theorem \ref{tmain} below).

Let us now describe our connectivity hypothesis, which says, roughly speaking, that from each point
$X\in \Omega$, there is local non-tangential access to an ample portion of a surface ball
at a scale on the order of $\delta(X):=\dist(X,\pom)$.  Let us make this a bit more precise.
A ``carrot path" (aka non-tangential path) joining
a point $X\in \om$, and a point $y \in \pom$, is a
connected rectifiable path $\gamma=\gamma(y,X)$, with endpoints $y$ and $X$,
such that
for some $\lambda\in(0,1)$ and for all $Z\in \gamma$,
\begin{equation}\label{eq1.2}
\lambda\, \ell\big(\gamma(y,Z)\big)\, \leq\,  \delta(Z)\,.
\end{equation} For $X\in \Omega$, and $R\geq 2$, set
$$\Delta_X = \Delta_X^R := B\big(X,R\delta(X)\big)\cap\pom\,.$$
We assume that  every
point $X\in \Omega$ may be joined by a carrot path to each $y$ in
a  ``Big Piece"  of $\Delta_X$, i.e., to each $y$
in a Borel subset $F \subset \Delta_X$, with $\sigma(F)\geq \theta\sigma(\Delta_X)$, where 
$\sigma$ denotes surface measure on $\pom$, and where the parameters
$R\geq 2$, $\lambda\in (0,1)$, and $\theta\in (0,1]$ are uniformly controlled.
We refer to this condition as a ``weak local John
condition", although ``weak local semi-uniformity" would be equally appropriate.
See Definitions \ref{def1.carrot}, \ref{def1.johnpoint}
and \ref{def1.john} for more details.    We remark that a strong version of the local John condition
(i.e., with $\theta =1$) has appeared in \cite{HMT}, in connection with boundary Poincar\'e inequalities 
for non-smooth domains.  

We observe that the weak local John condition is strictly weaker than semi-uniformity:
for example, the unit disk centered a the origin, with either the cross 
$\{-1/2\leq x\leq 1/2, y=0\}\cup \{-1/2\leq y\leq 1/2, x=0\}$ removed, or with the slit
$0\leq x\leq 1, y=0$ removed,
satisfies the weak local John condition, although semi-uniformity fails in each case.

The main result in the present work  is the following. The terminology used here will be defined in the sequel.  
\begin{theorem}\label{t1}  
	Let $\Omega \subset \ree$ be an open set, not necessarily connected, with an 
Ahlfors-David regular (ADR) boundary.
Then the following are equivalent:
\begin{itemize}
\item[(i)]
$\pom$ is uniformly rectifiable (UR; see Definition \ref{defur} below), and 
$\Omega$ satisfies the weak local John condition.
\smallskip
\item[(ii)]  $\Omega$
	satisfies an Interior Big Pieces of Chord-Arc Domains (IBPCAD) 
	condition (see Definition
	\ref{def1.ibpcad} below).
\end{itemize}	
\end{theorem}

Only the direction (i) implies (ii) is non-trivial.
For the converse,  the fact that IBPCAD implies the  
weak local John condition is immediate from the definitions.
Moreover, the boundary of a chord-arc domain is 
UR, and an ADR set with big pieces of UR is also UR (see \cite{DS2}). 

We note that condition (ii) (hence also condition (i)) implies that
harmonic measure is locally in weak-$A_\infty$ (see Definition \ref{deflocalAinfty})
with respect to surface measure:  this fact is well known, and follows 
from the maximum principle and the result 
of \cite{DJe} and \cite{Se} for chord-arc domains, and the method of \cite{BL}.

Moreover, for an open set with ADR boundary, the 
weak-$A_\infty$ property implies (and in fact is equivalent to) 
solvability of the Dirichlet problem for some $p<\infty$;
we refer the reader to, e.g., 
\cite[Section 4]{HLe} for details.
  We therefore have the following.

\begin{corollary}\label{c1}  Let $\Omega \subset \ree$ be an open set, not necessarily connected,
with a uniformly rectifiable boundary.  Suppose in addition that
$\om$ satisfies the weak local John condition.  Then the $L^p$
Dirichlet problem for $\Omega$ is solvable in $L^p$, for some $p<\infty$, i.e., given continuous data
$g$ defined on $\pom$, for the harmonic measure solution $u$ to the Dirichlet problem with data $g$, we have
for some $p<\infty$ that
\begin{equation}\label{eq1.Dirichlet}
\|N_*u\|_{L^{p}(\pom)}\, \leq \,C\, \|g\|_{L^p(\pom)}\,,
\end{equation}
where $N_*u$ is a suitable version of the non-tangential maximal function of $u$. 
\end{corollary}

Combining the previous results 
with certain other recent works (to be discussed momentarily), 
one obtains the following geometric characterization of
quantitative absolute continuity of harmonic measure, and of the $L^p$ solvability
of the Dirichlet problem. 
\begin{theorem}\label{tmain}
Let $\Omega\subset \ree$ be an open set satisfying an interior Corkscrew condition (see Definition \ref{def1.cork} below), and suppose that $\pom$ is Ahlfors-David
regular (ADR).   Then the following are equivalent:
\begin{enumerate}
\item $\pom$ is Uniformly Rectifiable (UR) and $\Omega$ satisfies a weak local John condition.
\item $\Omega$ satisfies an Interior Big Pieces of 
Chord-Arc Domains (IBPCAD) condition.
\item Harmonic measure $\hm$ is locally in weak-$A_\infty$ (see Definition \ref{deflocalAinfty} below)
with respect to surface measure $\sigma$ on $\pom$.

\item The $L^p$ Dirichlet problem is solvable in the sense of
Corollary  \ref{c1}, for some $p<\infty$. 
\end{enumerate}
\end{theorem}

Some comments are in order.   In the present paper we shall prove that (1) implies (2) (this is the content of Theorem \ref{t1}), and we note that the interior Corkscrew
condition is not needed for this particular implication (nor for (2) implies (3) 
if and only if (4)).   
Rather, it is crucial for the converse direction (3) implies (1). As noted above, the fact that
(2) implies (3) follows by a well-known argument using the maximum principle 
and the result of
\cite{DJe} and \cite{Se} for chord-arc domains, along with the criterion for
weak-$A_\infty$ obtained in \cite{BL}; the equivalence of (3) and (4) is well known. 
The implication (3) implies (1) has two parts: that weak-$A_\infty$ implies
UR is the main result of \cite{HM-4}; an alternative proof, with a more general result, appears in \cite{MT}, and see also
\cite{HLMN} for the final published version of the results of \cite{HM-4}, along with an extension to the $p$-harmonic setting.
The remaining implication, that weak-$A_\infty$ implies weak local John, is a very recent result of Azzam, Mourgoglou and
Tolsa \cite{AMT}.  We remark that in a preliminary version of this work  \cite{HM4}, we have given a direct proof that
(1) implies (3), using an approach similar to that of the present paper.  We also mention that our background hypotheses
(upper and lower ADR, and interior Corkscrew) are in the nature of best possible:  
one may construct a counter-example in
the absence of any one of them, for at least one direction of this chain of implications,
as we shall discuss in Appendix \ref{appa}.


\subsection{Further notation and definitions}

\begin{list}{$\bullet$}{\leftmargin=0.4cm  \itemsep=0.2cm}

\item Unless otherwise stated, we use the letters $c,C$ to denote harmless positive constants, not necessarily
the same at each occurrence, which depend only on dimension and the
constants appearing in the hypotheses of the theorems (which we refer to as the
``allowable parameters'').  We shall also
sometimes write $a\lesssim b$ and $a \approx b$ to mean, respectively,
that $a \leq C b$ and $0< c \leq a/b\leq C$, where the constants $c$ and $C$ are as above, unless
explicitly noted to the contrary.  At times, we shall designate by $M$ a particular constant whose value will remain unchanged throughout the proof of a given lemma or proposition, but
which may have a different value during the proof of a different lemma or proposition.

\item $\Omega$ will always denote an open set in $\ree$,
not necessarily connected unless otherwise specified.

\item  We use the notation
$\gamma(X,Y)$ to denote a rectifiable path with endpoints $X$ and $Y$,
and its arc-length will be denoted $\ell(\gamma(X,Y))$.  Given such a path,
if $Z\in \gamma(X,Y)$, we use the notation $\gamma(Z,Y)$ to denote the portion of the original path
with endpoints $Z$ and $Y$.  

\item Given an open set $\om \subset \ree$, we shall
use lower case letters $x,y,z$, etc., to denote points on $\pom$, and capital letters
$X,Y,Z$, etc., to denote generic points in $\om$ (or more generally in $\ree\setminus \pom$).

\item We let $e_j$,  $j=1,2,\dots,n+1,$ denote the standard unit basis vectors in $\ree$.

\item The open $(n+1)$-dimensional Euclidean ball of radius $r$ will be denoted
$B(x,r)$ when the center $x$ lies on $\pom$, or $B(X,r)$ when the center
$X \in \om$.  A {\it surface ball} is denoted
$\Delta(x,r):= B(x,r) \cap\partial\Omega.$

\item Given a Euclidean ball $B$ or surface ball $\Delta$, its radius will be denoted
$r_B$ or $r_\Delta$, respectively.

\item Given a Euclidean or surface ball $B= B(X,r)$ or $\Delta = \Delta(x,r)$, its concentric
dilate by a factor of $\kappa >0$ will be denoted
$\kappa B := B(X,\kappa r)$ or $\kappa \Delta := \Delta(x,\kappa r).$

\item Given an open set $\om \subset \ree$, for $X \in \om$, we set $\delta(X):= \dist(X,\pom)$.

\item We let $H^n$ denote $n$-dimensional Hausdorff measure, and let
$\sigma := H^n\big\lfloor_{\,\pom}$ denote the surface measure on $\pom$.

\item For a Borel set $A\subset \ree$, we let $1_A$ denote the usual
indicator function of $A$, i.e. $1_A(x) = 1$ if $x\in A$, and $1_A(x)= 0$ if $x\notin A$.

\item For a Borel set $A\subset \ree$,  we let $\interior(A)$ denote the interior of $A$.


\item Given a Borel measure $\mu$, and a Borel set $A$, with positive and finite $\mu$ measure, we
set $\fint_A f d\mu := \mu(A)^{-1} \int_A f d\mu$.

\item We shall use the letter $I$ (and sometimes $J$)
to denote a closed $(n+1)$-dimensional Euclidean dyadic cube with sides
parallel to the co-ordinate axes, and we let $\ell(I)$ denote the side length of $I$.
If $\ell(I) =2^{-k}$, then we set $k_I:= k$.
Given an ADR set $E\subset \ree$, we use $Q$ (or sometimes $P$)
to denote a dyadic ``cube''
on $E$.  The
latter exist (cf. \cite{DS1}, \cite{Ch}), and enjoy certain properties
which we enumerate in Lemma \ref{lemmaCh} below.

\end{list}

\begin{definition}\label{defadr} ({\bf  ADR})  (aka {\it Ahlfors-David regular}).
We say that a  set $E \subset \ree$, of Hausdorff dimension $n$, is ADR
if it is closed, and if there is some uniform constant $C$ such that
\begin{equation} \label{eq1.ADR}
\frac1C\, r^n \leq \sigma\big(\Delta(x,r)\big)
\leq C\, r^n,\quad\forall r\in(0,\diam (E)),\ x \in E,
\end{equation}
where $\diam(E)$ may be infinite.
Here, $\Delta(x,r):= E\cap B(x,r)$ is the {\it surface ball} of radius $r$,
and as above, $\sigma:= H^n\lfloor_{\,E}$ 
is the ``surface measure" on $E$.
\end{definition}

\begin{definition}\label{defur} ({\bf UR}) (aka {\it uniformly rectifiable}).
An $n$-dimensional ADR (hence closed) set $E\subset \ree$
is UR if and only if it contains ``Big Pieces of
Lipschitz Images" of $\rn$ (``BPLI").   This means that there are positive constants $c_1$ and
$C_1$, such that for each
$x\in E$ and each $r\in (0,\diam (E))$, there is a
Lipschitz mapping $\rho= \rho_{x,r}: \rn\to \ree$, with Lipschitz constant
no larger than $C_1$,
such that 
$$
H^n\Big(E\cap B(x,r)\cap  \rho\left(\{z\in\rn:|z|<r\}\right)\Big)\,\geq\,c_1 r^n\,.
$$
\end{definition}

We recall that $n$-dimensional rectifiable sets are characterized by the
property that they can be
covered, up to a set of
$H^n$ measure 0, by a countable union of Lipschitz images of $\rn$;
we observe that BPLI  is a quantitative version
of this fact.

We remark
that, at least among the class of ADR sets, the UR sets
are precisely those for which all ``sufficiently nice" singular integrals
are $L^2$-bounded  \cite{DS1}.    In fact, for $n$-dimensional ADR sets
in $\ree$, the $L^2$ boundedness of certain special singular integral operators
(the ``Riesz Transforms"), suffices to characterize uniform rectifiability (see \cite{MMV} for the case $n=1$, and 
\cite{NToV} in general). 
We further remark that
there exist sets that are ADR (and that even form the boundary of a domain satisfying 
interior Corkscrew and Harnack Chain conditions),
but that are totally non-rectifiable (e.g., see the construction of Garnett's ``4-corners Cantor set"
in \cite[Chapter1]{DS2}).  Finally, we mention that there are numerous other characterizations of UR sets
(many of which remain valid in higher co-dimensions); cf. \cite{DS1,DS2}.

\begin{definition}\label{defurchar} ({\bf ``UR character"}).   Given a UR set $E\subset \ree$, its ``UR character"
is just the pair of constants $(c_1,C_1)$ involved in the definition of uniform rectifiability,
along with the ADR constant; or equivalently,
the quantitative bounds involved in any particular characterization of uniform rectifiability.
\end{definition}

\begin{definition} ({\bf Corkscrew condition}).  \label{def1.cork}
Following
\cite{JK}, we say that an open set $\Omega\subset \ree$
satisfies the {\it Corkscrew condition} if for some uniform constant $c>0$ and
for every surface ball $\Delta:=\Delta(x,r),$ with $x\in \partial\Omega$ and
$0<r<\diam(\partial\Omega)$, there is a ball
$B(X_\Delta,cr)\subset B(x,r)\cap\Omega$.  The point $X_\Delta\subset \Omega$ is called
a {\it Corkscrew point} relative to $\Delta.$  We note that  we may allow
$r<C\diam(\pom)$ for any fixed $C$, simply by adjusting the constant $c$.
In order to emphasize that $B(X_\Delta,cr) \subset \om$, we shall sometimes refer to this property as
the  {\it interior
Corkscrew condition}.
\end{definition}

\begin{definition}({\bf Harnack Chains, and the Harnack Chain condition}  \cite{JK}).  \label{def1.hc} 
Given two points $X,X' \in \Omega$, and a pair of numbers $M,N\geq1$, 
an $(M,N)$-{\it Harnack Chain connecting $X$ to $X'$},  is a chain of
open balls
$B_1,\dots,B_N \subset \Omega$, 
with $X\in B_1,\, X'\in B_N,$ $B_k\cap B_{k+1}\neq \emptyset$
and $M^{-1}\diam (B_k) \leq \dist (B_k,\partial\Omega)\leq M\diam (B_k).$
We say that $\Omega$ satisfies the {\it Harnack Chain condition}
if there is a uniform constant $M$ such that for any two points $X,X'\in\om$,
there is an $(M,N)$-Harnack Chain connecting them, with $N$ depending only on the ratio
$|X-X'|/\left(\min\big(\delta(X),\delta(X')\big)\right)$.
\end{definition}

\begin{definition}({\bf NTA}). \label{def1.nta} Again following \cite{JK}, we say that a
domain $\Omega\subset \ree$ is NTA ({\it Non-tangentially accessible}) if it satisfies the
Harnack Chain condition, and if both $\Omega$ and
$\Omega_{\rm ext}:= \ree\setminus \overline{\Omega}$ satisfy the Corkscrew condition.
\end{definition}

\begin{definition}({\bf CAD}). \label{def1.cad}  We say that a connected open set $\om \subset \ree$
is a CAD ({\it Chord-arc domain}), if it is NTA, and if $\pom$ is ADR.
\end{definition}

\begin{definition}({\bf Carrot path}).  \label{def1.carrot} Let $\Omega\subset \ree$ be an open set.
Given a point $X\in \om$, and a point $y \in \pom$, we say that 
a connected rectifiable path $\gamma=\gamma(y,X)$, with endpoints $y$ and $X$,
is a {\it carrot path} (more precisely, a {\it $\lambda$-carrot path}) connecting
$y$ to $X$, if $\gamma\setminus\{y\}\subset \om$, and if 
for some $\lambda\in(0,1)$ and for all $Z\in \gamma$,
\begin{equation}\label{eq1.carrot}
\lambda\, \ell\big(\gamma(y,Z)\big)\, \leq\,  \delta(Z)\,.
\end{equation}
With a slight abuse of terminology, we shall sometimes 
refer to such a path as a {\it $\lambda$-carrot path in}
$\om$, although of course the endpoint $y$ lies on $\pom$.
\end{definition}

A carrot path is sometimes referred to as a non-tangential path.

\begin{definition}({\bf $(\theta,\lambda,R)$-weak local John point}).  \label{def1.johnpoint} Let $X\in\om$, 
and for  constants $\theta\in (0,1]$, $\lambda\in (0,1)$, and $R\geq 2$, set
$$\Delta_X=\Delta_X^R:= B\big(X, R\delta(X)\big)\cap\pom\,.$$
We say that a point $X\in\Omega$ is a {\it $(\theta,\lambda,R)$-weak local John point}
if there is a Borel set $F\subset \Delta^R_X$, with
$\sigma(F)\geq \theta\sigma(\Delta^R_X)$, such that for every $y\in F$, there is a 
$\lambda$-carrot path connecting $y$ to $X$.
\end{definition}

Thus, a weak local John point is non-tangentially connected to an ample portion of the boundary,
locally.  We observe that one can always choose $R$ smaller, for possibly different values of
$\theta$ and $\lambda$, by moving from $X$ to a point $X'$ on a line segment joining $X$ to the boundary.

\begin{remark}\label{remark2}
We observe that it is a slight abuse of notation to write
$\Delta_X$,
since the latter is not centered on $\pom$, and thus it is not a true surface ball; on the other hand,
there are true surface balls,  $\Delta'_X:=\Delta(\hat{x},(R-1)\delta(X))$ and 
$\Delta''_X:=\Delta(\hat{x},(R+1)\delta(X))$, centered at a ``touching point"
$\hat{x}\in\pom$ with $\delta(X)=|X-\hat{x}|$,
which, respectively, are contained in, and contain, $\Delta_X$.   
\end{remark}

\begin{definition}({\bf Weak local John condition}).  \label{def1.john}
We say that $\om$ satisfies
a {\it weak local  John condition} if there are constants $\lambda\in (0,1)$,
$\theta\in(0,1]$, and $R\geq 2$, such that every $X\in\om$ is a $(\theta,\lambda,R)$-weak local John point.
\end{definition}

\begin{definition}({\bf IBPCAD}). \label{def1.ibpcad} We say that a connected open set $\om \subset \ree$
has {\it Interior Big Pieces of Chord-Arc Domains} (IBPCAD) if there exist positive constants $\eta$ and $C$, 
and $R\geq 2$, such that for every $X\in \Omega$, with $\delta(X)<\diam(\pom)$,
there is a  chord-arc domain $\Omega_X\subset \Omega$ satisfying
\begin{itemize}
\item $X\in \Omega_X$.
\item $\dist(X,\pom_X) \geq \eta \delta(X)$.
\item $\diam(\Omega_X) \leq C\delta(X)$.
\item $\sigma(\pom_X\cap \Delta^R_X) \geq \,\eta\, \sigma(\Delta^R_X) \,\approx_R\, \eta\,\delta(X)^n$.
\item The chord-arc constants of the domains $\Omega_X$ are uniform in $X$.
\end{itemize}
\end{definition}

\begin{remark}\label{remarkBPCAD}
In the presence of an interior Corkscrew condition, Definition \ref{def1.ibpcad}
is easily seen to be equivalent to the following more standard ``Big Pieces" condition: there is a constant $\eta>0$
(perhaps slightly different to that in Definition \ref{def1.ibpcad}), such that
for each surface ball $\Delta:=\Delta(x,r) = B(x,r) \cap \pom$, $x\in \pom$ and
$r <\diam(\pom)$, there is a chord-arc domain $\Omega_\Delta$ satisfying
\begin{itemize}
\item $\Omega_\Delta\subset B(x,r) \cap\Omega$.
\item $\Omega_\Delta$ contains a Corkscrew point $X_\Delta$, with
 $\dist(X_\Delta,\pom_\Delta) \geq \eta r$.
\item $\sigma(\pom_\Delta\cap \Delta) \geq \,\eta\, \sigma(\Delta)\approx \eta r^n$.
\item The chord-arc constants of the domains $\Omega_\Delta$ are uniform in $\Delta$.
\end{itemize}
\end{remark}

\begin{definition}\label{defAinfty}
({\bf $A_\infty$}, weak-$A_\infty$, and weak-$RH_q$). 
Given an ADR set $E\subset\ree$, 
and a surface ball
$\Delta_0:= B_0 \cap E$ centered at $E$,
we say that a Borel measure $\mu$ defined on $E$ belongs to
$A_\infty(\Delta_0)$ if there are positive constants $C$ and $s$
such that for each surface ball $\Delta = B\cap E$ centered on $E$, with $B\subseteq B_0$,
we have
\begin{equation}\label{eq1.ainfty}
\mu (A) \leq C \left(\frac{\sigma(A)}{\sigma(\Delta)}\right)^s\,\mu (\Delta)\,,
\qquad \mbox{for every Borel set } A\subset \Delta\,.
\end{equation}
Similarly, we say that $\mu \in$ weak-$A_\infty(\Delta_0)$ if 
for each surface ball $\Delta = B\cap E$ centered on $E$, with $2B\subseteq B_0$,
\begin{equation}\label{eq1.wainfty}
\mu (A) \leq C \left(\frac{\sigma(A)}{\sigma(\Delta)}\right)^s\,\mu (2\Delta)\,,
\qquad \mbox{for every Borel set } A\subset \Delta\,.
\end{equation}
We recall that, as is well known, the condition $\mu \in$ weak-$A_\infty(\Delta_0)$
is equivalent to the property that $\mu \ll \sigma$ in $\Delta_0$, and that for some $q>1$, the
Radon-Nikodym derivative $k:= d\mu/d\sigma$ satisfies
the weak reverse H\"older estimate
\begin{equation}\label{eq1.wRH}
\left(\fint_\Delta k^q d\sigma \right)^{1/q} \,\lesssim\, \fint_{2\Delta} k \,d\sigma\,
\approx\,  \frac{\mu(2\Delta)}{\sigma(\Delta)}\,,
\quad \forall\, \Delta = B\cap E,\,\, {\rm with} \,\, 2B\subseteq B_0\,,
\end{equation}
with $B$ centered on $E$.
We shall refer to the inequality in \eqref{eq1.wRH} as
an  ``$RH_q$" estimate, and we shall say that $k\in RH_q(\Delta_0)$ if $k$ satisfies \eqref{eq1.wRH}.
\end{definition}

\begin{definition}\label{deflocalAinfty} ({\bf Local $A_\infty$ and local weak-$A_\infty$}).
We say that harmonic measure $\hm$  is locally in $A_\infty$ (resp., locally in
weak-$A_\infty$) on $\pom$,
if there are  uniform positive constants $C$ and $s$
such that for every ball $B=B(x,r)$ centered on $\pom$, 
with radius  $r<\diam(\pom)/4$, and associated surface ball $\Delta=B\cap\pom$,
\begin{equation}\label{eq1.localainfty}
\hm^X (A) \leq C \left(\frac{\sigma(A)}{\sigma(\Delta)}\right)^s\,\hm^X (\Delta)\,,
\qquad \forall \, X\in\om\setminus 4B\,, \,\,\forall \mbox{ Borel } A\subset \Delta\,,
\end{equation}
or, respectively, that
\begin{equation}\label{eq1.localwainfty}
\hm^X (A) \leq C \left(\frac{\sigma(A)}{\sigma(\Delta)}\right)^s\,\hm^X (2\Delta)\,,
\qquad \forall \, X\in\om\setminus 4B\,, \,\,\forall \mbox{ Borel } A\subset \Delta\,;
\end{equation}
equivalently,
if for every ball $B$ and surface ball $\Delta=B\cap\pom$ as above,
and for each  
point $X\in\om\setminus 4B$, $\hm^X\in$ $A_\infty(\Delta)$ (resp.,
$\hm^X\in$ weak-$A_\infty(\Delta)$) with uniformly controlled $A_\infty$ (resp., weak-$A_\infty$) constants. 
\end{definition}

\begin{lemma}\label{lemmaCh}\textup{({\bf Existence and properties of the ``dyadic grid''})
\cite{DS1,DS2}, \cite{Ch}.}
Suppose that $E\subset \ree$ is an $n$-dimensional ADR set.  Then there exist
constants $ a_0>0,\, s>0$ and $C_1<\infty$, depending only on $n$ and the
ADR constant, such that for each $k \in \mathbb{Z},$
there is a collection of Borel sets (``cubes'')
$$
\mathbb{D}_k:=\{Q_{j}^k\subset E: j\in \mathfrak{I}_k\},$$ where
$\mathfrak{I}_k$ denotes some (possibly finite) index set depending on $k$, satisfying

\begin{list}{$(\theenumi)$}{\usecounter{enumi}\leftmargin=.8cm
\labelwidth=.8cm\itemsep=0.2cm\topsep=.1cm
\renewcommand{\theenumi}{\roman{enumi}}}

\item $E=\cup_{j}Q_{j}^k\,\,$ for each
$k\in{\mathbb Z}$.

\item If $m\geq k$ then either $Q_{i}^{m}\subset Q_{j}^{k}$ or
$Q_{i}^{m}\cap Q_{j}^{k}=\emptyset$.

\item For each $(j,k)$ and each $m<k$, there is a unique
$i$ such that $Q_{j}^k\subset Q_{i}^m$.

\item $\diam\big(Q_{j}^k\big)\leq C_1 2^{-k}$.

\item Each $Q_{j}^k$ contains some ``surface ball'' $\Delta \big(x^k_{j},a_02^{-k}\big):=
B\big(x^k_{j},a_02^{-k}\big)\cap E$.

\item $H^n\big(\big\{x\in Q^k_j:{\rm dist}(x,E\setminus Q^k_j)\leq \vartheta \,2^{-k}\big\}\big)\leq
C_1\,\vartheta^s\,H^n\big(Q^k_j\big),$ for all $k,j$ and for all $\vartheta\in (0,a_0)$.
\end{list}
\end{lemma}

A few remarks are in order concerning this lemma.

\begin{list}{$\bullet$}{\leftmargin=0.4cm  \itemsep=0.2cm}

\item In the setting of a general space of homogeneous type, this lemma has been proved by Christ
\cite{Ch}, with the
dyadic parameter $1/2$ replaced by some constant $\delta \in (0,1)$.
In fact, one may always take $\delta = 1/2$ (cf.  \cite[Proof of Proposition 2.12]{HMMM}).
In the presence of the Ahlfors-David
property (\ref{eq1.ADR}), the result already appears in \cite{DS1,DS2}. Some predecessors of this construction have appeared in \cite{David88} and \cite{David91}.

\item  For our purposes, we may ignore those
$k\in \mathbb{Z}$ such that $2^{-k} \gtrsim {\rm diam}(E)$, in the case that the latter is finite.

\item  We shall denote by  $\mathbb{D}=\mathbb{D}(E)$ the collection of all relevant
$Q^k_j$, i.e., $$\mathbb{D} := \cup_{k} \mathbb{D}_k,$$
where, if $\diam (E)$ is finite, the union runs
over those $k$ such that $2^{-k} \lesssim  {\rm diam}(E)$.

\item Properties $(iv)$ and $(v)$ imply that for each cube $Q\in\mathbb{D}_k$,
there is a point $x_Q\in E$, a Euclidean ball $B(x_Q,r_Q)$ and a surface ball
$\Delta(x_Q,r_Q):= B(x_Q,r_Q)\cap E$ such that
$r_Q\approx 2^{-k} \approx {\rm diam}(Q)$
and \begin{equation}\label{cube-ball}
\Delta(x_Q,r_Q)\subset Q \subset \Delta(x_Q,Cr_Q),\end{equation}
for some uniform constant $C$.
We shall denote this ball and surface ball by
\begin{equation}\label{cube-ball2}
B_Q:= B(x_Q,r_Q) \,,\qquad\Delta_Q:= \Delta(x_Q,r_Q),\end{equation}
and we shall refer to the point $x_Q$ as the ``center'' of $Q$.

\item For a dyadic cube $Q\in \mathbb{D}_k$, we shall
set $\ell(Q) = 2^{-k}$, and we shall refer to this quantity as the ``length''
of $Q$.  Evidently, $\ell(Q)\approx \diam(Q).$

\item For a dyadic cube $Q \in \mathbb{D}$, we let $k(Q)$ denote the dyadic generation
to which $Q$ belongs, i.e., we set  $k = k(Q)$ if
$Q\in \mathbb{D}_k$; thus, $\ell(Q) =2^{-k(Q)}$.

\item For a pair of cubes $Q',Q \in \mathbb{D}$, if $Q'$ is a dyadic child of $Q$,
i.e., if $Q'\subset Q$, and $\ell(Q) =2\ell(Q')$, then we write $Q'\lhd Q$.

\end{list}

With the dyadic cubes in hand, we may now define the notion of a Corkscrew point relative to a cube $Q$.

\begin{definition}({\bf Corkscrew point relative to $Q$}).  \label{def1.CScube}
Let $\om$ satisfy the Corkscrew condition (Definition \ref{def1.cork}), suppose that $\pom$ is ADR,
and let $Q\in \dd(\pom)$.
A {\it Corkscrew point relative to $Q$} is simply a Corkscrew point relative to the surface ball
$\Delta_Q$ defined \eqref{cube-ball}-\eqref{cube-ball2}.
\end{definition}

\begin{definition}({\bf Coherency and Semi-coherency}). \cite{DS2}.
\label{d3.11}   
Let $E\subset \ree$ be an ADR set.
Let $\sbf\subset \dd(E)$. We say that $\sbf$ is
{\it coherent} if the following conditions hold:
\begin{itemize}\itemsep=0.1cm

\item[$(a)$] $\sbf$ contains a unique maximal element $Q(\sbf)$ which contains all other elements of $\sbf$ as subsets.

\item[$(b)$] If $Q$  belongs to $\sbf$, and if $Q\subset \widetilde{Q}\subset Q(\sbf)$, then $\widetilde{Q}\in {\bf S}$.

\item[$(c)$] Given a cube $Q\in \sbf$, either all of its children belong to $\sbf$, or none of them do.

\end{itemize}
We say that $\sbf$ is {\it semi-coherent} if conditions $(a)$ and $(b)$ hold. 
\end{definition}

\section{Preliminaries}\label{s2}

We begin by recalling a bilateral version of the
David-Semmes ``Corona decomposition" of a UR set.   We refer the reader to \cite{HMM} for the proof.

\begin{lemma}\label{lemma2.1}\textup{(\cite[Lemma 2.2]{HMM})}  
Let $E\subset \ree$ be a UR set of dimension $n$.  Then given any positive constants
$\eta\ll 1$
and $K\gg 1$, there is a disjoint decomposition
$\dd(E) = \G\cup\B$, satisfying the following properties.
\begin{enumerate}
\item  The ``Good" collection $\G$ is further subdivided into
disjoint stopping time regimes, such that each such regime {\bf S} is coherent (Definition \ref{d3.11}).

\smallskip
\item The ``Bad" cubes, as well as the maximal cubes $Q(\sbf)$, $\sbf\subset\G$, satisfy a Carleson
packing condition:
$$\sum_{Q'\subset Q, \,Q'\in\B} \sigma(Q')
\,\,+\,\sum_{\sbf\subset\G: Q(\sbf)\subset Q}\sigma\big(Q(\sbf)\big)\,\leq\, C_{\eta,K}\, \sigma(Q)\,,
\quad \forall Q\in \dd(E)\,.$$

\smallskip
\item For each $\sbf\subset\G$, there is a Lipschitz graph $\Gamma_{\sbf}$, with Lipschitz constant
at most $\eta$, such that, for every $Q\in \sbf$,
\begin{equation}\label{eq2.2a}
\sup_{x\in \Delta_Q^*} \dist(x,\Gamma_{\sbf} )\,
+\,\sup_{y\in B_Q^*\cap\Gamma_{\sbf}}\dist(y,E) < \eta\,\ell(Q)\,,
\end{equation}
where $B_Q^*:= B(x_Q,K\ell(Q))$ and $\Delta_Q^*:= B_Q^*\cap E$, and $x_Q$ is the ``center"
of $Q$ as in \eqref{cube-ball}-\eqref{cube-ball2}.
\end{enumerate}
\end{lemma}

We mention that David and Semmes, in \cite{DS1}, had
previously proved a unilateral version of Lemma \ref{lemma2.1}, 
in which the bilateral estimate \eqref{eq2.2a} 
is replaced by the unilateral bound
\begin{equation}\label{eq2.5}
\sup_{x\in \Delta_Q^*} \dist(x,\Gamma_{\sbf} )\,
<\, \eta\,\ell(Q)\,,\qquad \forall\,Q\in \sbf\,.
\end{equation} 

Next, we make a standard Whitney decomposition of $\Omega_E:=\ree\setminus E$, for a given UR set $E$
(in particular, 
$\om_E$ is open, since UR sets are closed by definition).
Let $\mathcal{W}=\W(\om_E)$ denote a collection
of (closed) dyadic Whitney cubes of $\om_E$, so that the cubes in $\mathcal{W}$
form a pairwise non-overlapping covering of $\om_E$, which satisfy
\begin{equation}\label{Whintey-4I}
4 \diam(I)\leq
\dist(4I,\pom)\leq \dist(I,\pom) \leq 40\diam(I)\,,\qquad \forall\, I\in \mathcal{W}\,\end{equation}
(just dyadically divide the standard Whitney cubes, as constructed in  \cite[Chapter VI]{St},
into cubes with side length 1/8 as large)
and also
$$(1/4)\diam(I_1)\leq\diam(I_2)\leq 4\diam(I_1)\,,$$
whenever $I_1$ and $I_2$ touch.

We fix a small parameter $\tau_0>0$, so that
for any $I\in \W$, and any $\tau \in (0,\tau_0]$,
the concentric dilate
\begin{equation}\label{whitney1}
I^*(\tau):= (1+\tau) I
\end{equation} 
still satisfies the Whitney property
\begin{equation}\label{whitney}
\diam I\approx \diam I^*(\tau) \approx \dist\left(I^*(\tau), E\right) \approx \dist(I,E)\,, \quad 0<\tau\leq \tau_0\,.
\end{equation}
Moreover,
for $\tau\leq\tau_0$ small enough, and for any $I,J\in \W$,
we have that $I^*(\tau)$ meets $J^*(\tau)$ if and only if
$I$ and $J$ have a boundary point in common, and that, if $I\neq J$,
then $I^*(\tau)$ misses $(3/4)J$.


Pick two parameters $\eta\ll 1$ and $K\gg 1$ (eventually, we shall take
$K=\eta^{-3/4}$).   For $Q\in \dd(E)$, define
\begin{equation}\label{eq3.1}
\W^0_Q:= \left\{I\in \W:\,\eta^{1/4} \ell(Q)\leq \ell(I)
 \leq K^{1/2}\ell(Q),\ \dist(I,Q)\leq K^{1/2} \ell(Q)\right\}.
 \end{equation}
 
\begin{remark}\label{remark:E-cks} 
We note that $\W^0_Q$ is non-empty,
 provided that we choose $\eta$ small enough, and $K$ large enough, depending only on dimension and ADR,
since the ADR condition implies that $\om_E$ satisfies a Corkscrew condition.  In the sequel, we shall always
assume that $\eta$ and $K$ have been so chosen.
 \end{remark}
 
Next, we recall a construction in  \cite[Section 3]{HMM}, leading up to and including in particular
\cite[Lemma 3.24]{HMM}.   We summarize this construction as follows. 
\begin{lemma}\label{lemma2.7}
Let $E\subset \ree$ be 
$n$-dimensional UR, and set $\om_E:= \ree\setminus E$.  Given positive constants
$\eta\ll 1$
and $K\gg 1$, as in \eqref{eq3.1} and Remark \ref{remark:E-cks},  
let $\dd(E) = \G\cup\B$, be the corresponding 
bilateral Corona decomposition of Lemma \ref{lemma2.1}. 
Then for each $\sbf\subset \G$, and for each $Q\in \sbf$, the collection 
$\W^0_Q$ in \eqref{eq3.1} has an augmentation $\W^*_Q\subset \W$ satisfying the following properties.
\begin{enumerate}
\item $\W^0_Q\subset \W^*_Q = \W_Q^{*,+}\cup\W_Q^{*,-}$,
where (after a suitable rotation of coordinates)
each $I \in \W_Q^{*,+}$ lies above the Lipschitz graph $\Gamma_{\sbf}$
of Lemma \ref{lemma2.1},  each $I \in \W_Q^{*,-}$ lies below $\Gamma_{\sbf}$.
Moreover, if $Q'$ is a child of $Q$, also belonging to
$\sbf$, then $\W_Q^{*,+}$ (resp. $\W_Q^{*,-}$) belongs to the same connected
component of
$\om_E$ as does $\W_{Q'}^{*,+}$ (resp. $\W_{Q'}^{*,-}$)
and  $\W_{Q'}^{*,+}\cap \W_{Q}^{*,+}\neq \emptyset$ (resp.,
$\W_{Q'}^{*,-}\cap\W_{Q}^{*,-}\neq \emptyset$). 

\smallskip
\item There are uniform constants $c$ and $C$ such that
\begin{equation}\label{eq2.whitney2}
\begin{array}{c}
c\eta^{1/2} \ell(Q)\leq \ell(I) \leq CK^{1/2}\ell(Q)\,, \quad \forall I\in \mathcal{W}^*_Q,
\\[5pt]
\dist(I,Q)\leq CK^{1/2} \ell(Q)\,,\quad\forall I\in \mathcal{W}^*_Q,
\\[5pt]
c\eta^{1/2} \ell(Q)\leq\dist(I^*(\tau),\Gamma_{\sbf})\,,\quad \forall I\in \mathcal{W}^*_Q\,,\quad \forall 
\tau\in (0,\tau_0]\,.
\end{array}
\end{equation}
\end{enumerate}

Moreover, given $\tau\in(0,\tau_0]$, set
\begin{equation}\label{eq3.3aa}
U^\pm_Q=U^\pm_{Q,\tau}:= \bigcup_{I\in \W^{*,\pm}_Q} {\rm int}\left(I^*(\tau)\right)\,,\qquad U_Q:= U_Q^+\cup U_Q^-\,,
\end{equation}
and given $\sbf'$, a semi-coherent subregime of $\sbf$, define 
\begin{equation}\label{eq3.2}
\Omega_{\sbf'}^\pm = \Omega_{\sbf'}^\pm(\tau) := \bigcup_{Q\in\sbf'} U_Q^{\pm}\,.
\end{equation}
Then 
each of $\Omega^\pm_{\sbf'}$ is a CAD, with Chord-arc constants
depending only on $n,\tau,\eta, K$, and the
ADR/UR constants for $\pom$.  
\end{lemma}


\begin{remark}\label{remark2.12} In particular, for each $\sbf\subset \G$,
if $Q'$ and $Q$ belong to $\sbf$, and 
if $Q'$ is a dyadic child of $Q$, then $U_{Q'}^+\cup U_{Q}^+$ is Harnack Chain connected,
and every pair of points $X,Y\in U_{Q'}^+\cup U_Q^+$ may be connected by a Harnack Chain 
 in $\Omega_E$ 
of length at most $C= C(n,\tau,\eta,K,\textup{ADR/UR})$.  The same is true for  
$U_{Q'}^-\cup U_{Q}^-$.
\end{remark}

\begin{remark}\label{remark2.13} Let $0<\tau\leq \tau_0/2$.
Given any $\sbf\subset \G$, and any semi-coherent subregime
$\sbf'\subset \sbf$, define $\om_{\sbf'}^\pm=\om_{\sbf'}^\pm(\tau)$ as in \eqref{eq3.2},
and similarly set $\widehat{\om}_{\sbf'}^\pm=\om_{\sbf'}^\pm(2\tau)$.  Then by construction, for
any $X\in \overline{\om_{\sbf'}^\pm}$, 
$$ 
\dist(X,E) \approx \dist(X, \partial \widehat{\om}_{\sbf'}^\pm)\,,$$
where of course the implicit constants depend on $\tau$.
\end{remark}

As in \cite{HMM}, it will be useful for us to extend the definition of the Whitney region $U_Q$ to the case that
$Q\in\B$, the ``bad" collection of Lemma \ref{lemma2.1}.   Let $\W_Q^*$ be the augmentation of $\W_Q^0$
as constructed in Lemma \ref{lemma2.7}, and set
\begin{equation}\label{Wdef}
\W_Q:=\left\{
\begin{array}{l}
\W_Q^*\,,
\,\,Q\in\G,
\\[6pt]
\W_Q^0\,,
\,\,Q\in\B
\end{array}
\right.\,.
\end{equation}
For $Q \in\G$ we shall henceforth simply write $\W_Q^\pm$ in place of $\W_Q^{*,\pm}$.
For arbitrary $Q\in \dd(E)$, good or bad, we may then define
\begin{equation}\label{eq3.3bb}
U_Q=U_{Q,\tau}:= \bigcup_{I\in \W_Q} {\rm int}\left(I^*(\tau)\right)\,.
\end{equation}
Let us note that for $Q\in\G$, the latter definition agrees with that in \eqref{eq3.3aa}. Note that by construction
\begin{equation}\label{dist:UQ-pom}
U_Q\subset\{Y\in\Omega:\ \dist(Y,\pom)> c\eta^{1/2}\ell(Q)\}\cap B(x_Q, CK^{1/2}\ell(Q)),
\end{equation}
for some uniform constants $C\ge 1$ and $0<c<1$ (see \eqref{Whintey-4I}, \eqref{eq3.1}, and \eqref{eq2.whitney2}).
In particular, for every $Q\in\dd$ if follows that
\begin{equation}\label{def:BQ*}
\bigcup_{Q'\in\dd_Q} U_{Q'}
\subset B(x_Q,K\ell(Q))=:B_Q^*.
\end{equation}


For future reference, we introduce dyadic sawtooth regions as follows.  Set
\begin{equation}\label{eq3.4a}
\dd_Q:=\left\{Q'\in\dd(E):Q'\subset Q\right\}\,,
\end{equation}
and given $k\ge 1$,
\begin{equation}\label{eq3.4aaa}
\dd_Q^k:=\left\{Q'\in\dd(E):Q'\subset Q, \ \ell(Q')=2^{-k}\,\ell(Q)\right\}\,,
\end{equation}
Given a family $\mathcal{F}$ of disjoint cubes $\{Q_j\}\subset \mathbb{D}$, we define
the {\it global discretized sawtooth} relative to $\F$ by
\begin{equation}\label{eq2.discretesawtooth1}
\dd_{\F}:=\dd\setminus \bigcup_{Q_j\in \F} \dd_{Q_j}\,,
\end{equation}
i.e., $\dd_{\F}$ is the collection of all $Q\in\dd$ that are not contained in any $Q_j\in\F$.
 We may allow $\F$ to be empty, in which case $\dd_\F=\dd$. 
Given some fixed cube $Q$,
the {\it local discretized sawtooth} relative to $\F$ by
\begin{equation}\label{eq2.discretesawtooth2}
\dd_{\F,Q}:=\dd_Q\setminus \bigcup_{Q_j\in \F} \dd_{Q_j}=\dd_\F\cap\dd_Q.
\end{equation}
Note that with this convention, $\dd_Q=\dd_{\textup{\O},Q}$ (i.e., if one takes $\F=\emptyset$
in \eqref{eq2.discretesawtooth2}).


\section{Proof of Theorem \ref{t1}}\label{s3}

In the proof of  Theorem \ref{t1}, we shall employ a two-parameter
induction argument, which is a refinement of the
method of ``extrapolation" of Carleson measures.   The latter is a bootstrapping scheme for
lifting the Carleson measure constant, developed by J. L. Lewis \cite{LM}, and based on
the corona construction of Carleson \cite{Car} and Carleson and Garnett \cite{CG}
(see also \cite{HL}, \cite{AHLT}, \cite{AHMTT}, \cite{HM-TAMS}, \cite{HM-I},\cite{HMM}).

\subsection{Step 1: the set-up}\label{ss3.1}
To set the stage for the induction procedure, let us begin by making a preliminary reduction.
It will be convenient to work with a certain dyadic version of  Definition \ref{def1.ibpcad}.
To this end,
 let $X\in \om$, with $\delta(X) < \diam(\pom)$, and set
$\Delta_X=\Delta_X^R=B(X,R\delta(X))\cap\pom$, for some fixed $R\geq 2$ as in Definition
\ref{def1.johnpoint}.
Let $\hat{x}\in\pom$ be a touching point for $X$, i.e., $|X-\hat{x}|=\delta(X)$. 
Choose $X_1$ on the line segment joining $X$ to $\hat{x}$, with
$\delta(X_1) = \delta(X)/2$, and set 
$\Delta_{X_1}=B(X_1,R\delta(X)/2)\cap \pom$.  Note that $B(X_1,R\delta(X)/2)\subset
B(X,R\delta(X))$, and furthermore,
$$\dist\Big(B(X_1,R\delta(X)/2), \partial B(X,R\delta(X)\Big) > \frac{R-1}{2}\delta(X) \geq \frac12\delta(X).$$
We may therefore 
cover $\Delta_{X_1}$ by a disjoint collection
$\{Q_i\}_{i=1}^N\subset\dd(\pom)$, of equal length $\ell(Q_i)\approx \delta(X)$, 
such that each $Q_i\subset \Delta_X$, and such that
the implicit constants depend only on $n$ and ADR, and thus 
the cardinality $N$ 
of the collection depends on $n$, ADR, and $R$. 
With $E=\pom$, we make the Whitney decomposition of the set
$\om_E =\ree\setminus E$ as in Section \ref{s2} (thus, $\om\subset\om_E$). Moreover, for
sufficiently small $\eta$ and sufficiently large $K$ in \eqref{eq3.1},
we then have that $X\in U_{Q_i}$ for each $i=1,2,\dots,N$.  By hypothesis, there are constants
$\theta_0\in (0,1],\lambda_0\in (0,1)$, and $R\geq 2$ as above, such that every $Z\in\om$ is a
$(\theta_0,\lambda_0,R)$-weak local John point (Definition \ref{def1.johnpoint}).  
In particular, this is true for
$X_1$, hence there is a Borel set $F\subset \Delta_{X_1}$, with
$\sigma(F) \geq \theta_0 \sigma(\Delta_{X_1})$, such that every
$y\in F$ may be connected to $X_1$ via a $\lambda_0$-carrot path.
By ADR, $\sigma(\Delta_{X_1})\approx \sum_{i=1}^N\sigma(Q_i)$
and thus by pigeon-holing, there is at least one $Q_i=:Q$ such that
$\sigma(F\cap Q) \geq \theta_1\sigma(Q)$, with $\theta_1$ depending only on
$\theta_0$, $n$ and ADR.  Moreover, the $\lambda_0$-carrot path connecting
each $y\in F$ to $X_1$ may be extended to a $\lambda_1$-carrot path
connecting $y$ to $X$, where $\lambda_1$ depends only on $\lambda_0$.

We have thus reduced matters to the following dyadic scenario:
let $Q\in\dd(\pom)$,  let $U_Q=U_{Q,\tau}$ be the associated Whitney region as in  \eqref{eq3.3bb},
with $\tau \leq \tau_0/2$ fixed, and suppose that $U_Q$ meets $\om$
(recall that by construction $U_Q\subset \om_E=\ree\setminus E$, with $E=\pom$).
For $X\in U_Q\cap\om$, and for a constant $\lambda \in (0,1)$, let
\begin{equation}\label{eqdefFcarQ}
F_{car}(X,Q)= F_{car}(X,Q,\lambda) 
\end{equation}
denote the set of $y\in Q$ which may be joined to $X$ by a $\lambda$-carrot path
$\gamma(y,X)$, and for $\theta \in (0,1]$, set
\begin{equation}\label{eqdefTQ}
T_Q=T_Q(\theta,\lambda):= \left\{ X\in U_Q\cap\om: \, \sigma\big(F_{car}(X,Q,\lambda)\big)\geq\theta\sigma(Q)\right\}.
\end{equation}
\begin{remark}\label{remark3.5}
Our goal is to prove that, given $\lambda \in (0,1)$ and $\theta\in (0,1]$,
there are
positive constants $\eta$ and $C$, depending on $\theta,\lambda$, and the 
allowable parameters, such that
for each  $Q\in\dd(\pom)$, and for each $ X\in T_Q(\theta,\lambda)$, there is a chord-arc domain 
$\Omega_X$, constructed as a union $\cup_kI_k^*$ of fattened Whitney boxes $I_k^*$, 
with uniformly controlled chord-arc constants, such that 
$$U^i_Q\subset\Omega_X\subset \Omega \cap B\big(X,C\delta(X)\big)\,,$$ 
where $U_Q^i$ is the particular connected component of $U_Q$ containing $X$, and
\begin{equation}\label{eq3.1a}
\sigma(\pom_X\cap Q) \geq \eta \sigma(Q)\,.
\end{equation}   
For some $Q\in \dd(\pom)$, 
it may be that $T_Q$ is empty.  On the other hand,
by the preceding discussion, each $X\in \om$ belongs to $T_Q(\theta_1,\lambda_1)$ for suitable
$Q,\theta_1$ and $\lambda_1$, so that 
 \eqref{eq3.1a} (with $\theta=\theta_1, \lambda=\lambda_1$) implies 
\begin{equation*}
\sigma(\pom_X\cap \Delta_X) \geq \eta_1 \sigma(\Delta_X)\,,
\end{equation*}
with $\eta_1\approx \eta$, where $Q$ is the particular $Q_i$ selected 
in the previous paragraph.  Moreover, since $X\in T_Q\subset U_Q$, we can modify $\Omega_X$
if necessary, by adjoining to it one or more fattened Whitney boxes $I^*$ with $\ell(I) \approx \ell(Q)$, to ensure that
for the modified $\om_X$, we have in addition that
$\dist(X,\pom_X) \gtrsim \ell(Q) \approx \delta(X)$, and therefore
$\om_X$ verifies all the conditions in Definition \ref{def1.ibpcad}.
\end{remark}

The rest of this section is therefore devoted to proving that there exists, for a given $Q$ and for each
$X\in T_Q(\theta,\lambda)$, a chord-arc domain $\Omega_X$ satisfying the stated properties
(when the set $T_Q(\theta,\lambda)$ is not vacuous).   To this end,
we let $\lambda \in(0,1)$
(by Remark \ref{remark3.5}, any fixed $\lambda \leq \lambda_1$ will suffice).
We also fix  positive numbers 
$K\gg \lambda^{-4}$, and $\eta \leq K^{-4/3}\ll \lambda^4$,
and for these values of $\eta$ and $K$, 
we make
the bilateral Corona decomposition of Lemma \ref{lemma2.1}, so that  $\dd(\pom)=\G\cup\B$. 
We also construct the Whitney collections $\W^0_Q$ in \eqref{eq3.1}, and $\W_Q^*$ of Lemma \ref{lemma2.7}
for this same choice of $\eta$ and $K$.


Given a cube $Q\in \dd(\pom)$, we set
\begin{equation}\label{eq5.8aa}
\dd_*(Q):=\left\{Q'\subset Q:\, \ell(Q)/4\leq \ell(Q')\leq \ell(Q)
\right\}\,.
\end{equation}
Thus, $\dd_*(Q)$ consists of the cube $Q$ itself, along with
its dyadic children and grandchildren.
Let
$$\M:=\{Q(\sbf)\}_{\sbf}$$ denote the collection of   
cubes which are the maximal elements of the stopping time regimes in $\G$.
We define
\begin{equation}\label{eq4.0}
\alpha_Q:= 
\begin{cases} \sigma(Q)\,,&{\rm if}\,\,  (\M\cup\B)\cap \dd_*(Q)\neq \emptyset, \\
0\,,& {\rm otherwise}.\end{cases}
\end{equation}
Given  any collection $\dd'\subset\dd(\pom)$, we set
\begin{equation}\label{eq4.1}
\mut(\dd'):= \sum_{Q\in\dd'}\alpha_{Q}.
\end{equation}
Then $\mut$ is a discrete Carleson measure, i.e.,
recalling that $\dd_Q$ is the discrete Carleson region
relative to $Q$
defined in \eqref{eq3.4a}, we claim that there is a uniform constant $C$ such that
\begin{equation}\label{eq6.2}
\mut(\dd_{Q})\, =\sum_{Q'\subset Q}\alpha_{Q'} \leq\, C\sigma(Q)\,,\qquad \forall\,Q\in \dd(\pom)\,.
\end{equation}
Indeed, note that
for any $Q'\in\dd_Q$, there are at most 3 cubes $Q$ such that 
$Q'\in \dd_*(Q)$ (namely, $Q'$ itself, its dyadic parent, and its dyadic grandparent),
and that by ADR,
$\sigma(Q)\approx \sigma(Q')$, if $Q'\in\dd_*(Q)$.   Thus,
given any
$Q_0\in\dd(\pom)$,
\begin{multline*}
\mut(\dd_{Q_0})\, =\sum_{Q\subset Q_0}\alpha_Q\,
\leq \sum_{Q'\in \M\cup\B}\,\sum_{Q\subset Q_0:\, Q'\in \dd_*(Q)}\sigma(Q) \\[4pt]
\lesssim\sum_{Q'\in \M\cup\B:\, Q' \subset Q_0}\sigma(Q')\,\leq\,C \sigma(Q_0)\,,
\end{multline*}
by Lemma \ref{lemma2.1} (2).  Here, and throughout the remainder of this section, a
generic constant $C$, and implicit constants, are allowed to depend upon the choice of the parameters
$\eta$ and $K$ that we have fixed, along with the usual allowable parameters.

With \eqref{eq6.2} in hand, we therefore have
\begin{equation}\label{eq4.7a}
M_0:= 
\sup_{Q\in\dd(E)}\frac{\mut(\dd_Q)}{\sigma(Q)}\leq C<\infty\,.
\end{equation}

As mentioned above, our proof will be based on a 
two parameter induction scheme.  Given $\lambda\in (0,\lambda_1]$ fixed as above, we 
recall that the set $F_{car}(X,Q,\lambda)$
is defined in \eqref{eqdefFcarQ}.   The induction hypothesis, which we formulate for any $a\geq 0$, and any $\theta\in (0,1]$ is as follows:
\\[.3cm]
\null\hskip.1cm \fbox{\rule[8pt]{0pt}{0pt}$H[a,\theta]$}\hskip4pt \fbox{\ \parbox[c]{.75\textwidth}{%
%
\rule[10pt]{0pt}{0pt}\it  
There is a positive constant $c_{a}=c_a(\theta)<1$ 
such that for any given $Q\in\dd(\pom)$,  if
\begin{equation}\label{eq3.12}
\!\!\!\!\!\!\!\!\!\!\!\!\!\!\!\!\!\!\!\!\!\!\!\!\!\!\!\!\!\!\!\!\!\!\!\!\!\!\mut(\dd_Q)\le \, a\sigma(Q),
\end{equation} 
and if there is a subset $V_Q\subset U_Q\cap\om$ for which 
\begin{equation}\label{eq3.10}
\!\!\!\!\!\!\!\!\!\!\!\!\!\!\!\!\!\!\!\!\!\!\!\!\!\!\!\!\!\!\!\!\!\!\!\!\!
\sigma\left(\bigcup_{X\in V_Q}F_{car}(X,Q,\lambda) \right) \, \geq\, \theta \sigma(Q)\,,
\end{equation}
then there is a subset $V^*_Q\subset V_Q$, 
such that for each connected component $U_Q^i$ of $U_Q$ 
which meets $V^*_Q$, there is a  chord-arc domain $\Omega_Q^i$ 
which is the interior of the union of a collection of fattened Whitney cubes $I^*$, and whose chord-arc 
constants depend only on dimension, $\lambda$, $a$, $\theta$, and the ADR constants for $\Omega$.
Moreover, 
$U_Q^i\subset \Omega^i_Q\subset B_Q^*\cap\Omega=B(x_Q,K\ell(Q))\cap\Omega$,
and
$\sum_i \sigma (\pom^i_{Q}\cap Q)\geq  c_a\sigma(Q)$,
where the sum runs over those $i$ such that $U_Q^i$ meets $V^*_Q$.
}\ }

\medskip

Let us briefly sketch the strategy of the proof.  We first fix $\theta=1$, and by induction on
$a$, establish $H[M_0,1]$.   We then show that there is a fixed $\zeta\in (0,1)$ such that
$H[M_0,\theta]$ implies $H[M_0,\zeta\theta]$, for every $\theta\in(0,1]$.   Iterating,
we then obtain $H[M_0,\theta_1]$ for any $\theta_1\in (0,1]$.  Now, by \eqref{eq4.7a},
we have \eqref{eq3.12} with $a=M_0$, for every $Q\in\dd(\pom)$.  Thus, $H[M_0,\theta_1]$ may be applied in
every cube $Q$ such that $T_Q(\theta_1,\lambda)$ (see \eqref{eqdefTQ}) is non-empty, with $V_Q=\{X\}$,
for any $X\in T_Q(\theta_1,\lambda)$.  
For $\lambda\leq \lambda_1$, and an appropriate choice
of $\theta_1$, by Remark \ref{remark3.5}, we obtain the existence of a chord-arc domain $\Omega_X$ 
verifying the conditions of Definition \ref{def1.ibpcad}, and thus that
Theorem \ref{t1} holds, as desired.

We begin with some preliminary observations.  In what follows we have fixed 
$\lambda \in(0,\lambda_1]$ and two positive numbers 
$K\gg \lambda^{-4}$, and $\eta \leq K^{-4/3}\ll \lambda^4$, for which the bilateral Corona decomposition of $\dd(\pom)$ in Lemma \ref{lemma2.1} is applied. We now fix $k_0\in\mathbb{N}$, $k_0\ge 4$,  such that
\begin{equation}\label{eq3.15}
2^{-k_0} \,\leq\, \frac{\eta}{K}\, < \,2^{-k_0+1}\,.
\end{equation}

\begin{lemma}\label{lemma3.15}  Let $Q\in\dd(\pom)$, and suppose that $Q'\subset Q$, with
$\ell(Q')\leq 2^{-k_0}\ell(Q)$.  Suppose that there are points $X\in U_Q\cap\om$ and $y\in Q'$, that are connected by a $\lambda$-carrot
path $\gamma=\gamma(y,X)$ in $\om$.   Then $\gamma$ meets $U_{Q'}\cap\Omega$. 
\end{lemma}

\begin{proof}  
By construction (see \eqref{eq3.1}, Lemma \ref{lemma2.7}, \eqref{Wdef} and \eqref{eq3.3bb}), 
$X\in U_Q$ implies that
$$\eta^{1/2}\ell(Q) \lesssim \delta(X) \lesssim K^{1/2}\ell(Q)\,.$$  Since $2^{-k_0}\ll\eta$, 
and $\ell(Q')\leq 2^{-k_0}\ell(Q)$, we then have that
$X\in \om\setminus B\big(y,2\ell(Q')\big)$, so $\gamma(y,X)$ meets
$B\big(y,2\ell(Q')\big)\setminus B\big(y,\ell(Q')\big)$, say at a point $Z$.  Since
$\gamma(y,X)$ is a $\lambda$-carrot path, and since we have previously
specified that $\eta \ll\lambda^4$,
$$\delta(Z) \geq \lambda \ell\big(\gamma(y,Z)\big) \geq \lambda|y-Z|\geq\lambda \ell(Q')\gg 
\eta^{1/4}\ell(Q')\,.$$  On the other hand
$$\delta(Z)\leq \dist(Z,Q') \leq |Z-y| \leq 2\ell(Q') \ll K^{1/2} \ell(Q')\,.$$
In particular then, the Whitney box $I$ containing $Z$ must belong to $\W^0_{Q'}$
(see \eqref{eq3.1}), so $Z\in U_{Q'}$.  Note that $Z\in\Omega$ since $\gamma\subset\Omega$.
\end{proof}

We shall also require the following.   We recall 
that by Lemma \ref{lemma2.7}, for $Q\in \sbf$, the Whitney region $U_Q$
has the splitting $U_Q = U_Q^+\cup U_Q^-$, with $U_Q^+$ (resp. $U_Q^-$) lying
above (resp., below) the Lipschitz graph $\Gamma_\sbf$ of Lemma \ref{lemma2.1}.

\begin{lemma}\label{lemma3.37}  Let $Q'\subset Q$, 
and suppose that $Q'$ and $Q$ both belong to $\G$, and moreover
that both $Q'$ and $Q$ belong to the same 
stopping time regime $\sbf$.
Suppose that $y\in Q'$ and $X\in U_Q\cap\om$ are connected via a $\lambda$-carrot path
$\gamma(y,X)$ in $\om$,
and assume that there is a point $Z\in \gamma(y,X) \cap U_{Q'}\cap\Omega$ 
(by Lemma \ref{lemma3.15} we know that such a $Z$ exists provided $\ell(Q')\leq 2^{-k_0} \ell(Q)$).
Then $X\in U_Q^+$ if and only if $Z\in U_{Q'}^+$ 
(thus, $X\in U_Q^-$  if and only if $Z\in U_{Q'}^-$).
\end{lemma}

\begin{proof}
We suppose for the sake of contradiction that, e.g., $X\in U_Q^+$, and that $Z\in U_{Q'}^-$.
Thus, in traveling from $y$ to $Z$ and then to $X$ along the path
$\gamma(y,X)$, one must cross the Lipschitz graph $\Gamma_\sbf$ at least once between
$Z$ and $X$.
Let $Y_1$ be the first point on $\gamma(y,X)\cap\Gamma_\sbf$ that one 
encounters {\it after} $Z$, when traveling toward $X$.
By Lemma \ref{lemma2.7},
\begin{equation*}
K^{1/2}\ell(Q) \gtrsim \delta(X) \geq \lambda \ell\big(\gamma(y,X)\big)
\gg K^{-1/4} \ell\big(\gamma(y,X)\big)\,,
\end{equation*}
where we recall that we have fixed $K\gg \lambda^{-4}$.
Consequently, $\ell\big(\gamma(y,X)\big) \ll K^{3/4} \ell(Q)$, so in particular,
$\gamma(y,X)\subset B_Q^*:= B\big(x_Q,K\ell(Q)\big)$, as in Lemma \ref{lemma2.1}.
On the other hand, $Y_1\notin B^*_{Q'}$.  Indeed,
$Y_1\in\Gamma_S$, so if $Y_1 \in B^*_{Q'}$, then by
\eqref{eq2.2a},  $\delta(Y_1) \leq \eta \ell(Q')$.  However,
$$\delta(Y_1) \geq \lambda \ell\big(\gamma(y,Y_1)\big)
\geq \lambda \ell\big(\gamma(y,Z)\big) \geq \lambda |y-Z| 
\geq \lambda \dist(Z,Q')  \gtrsim \lambda \eta^{1/2} \ell(Q')\,,$$
where in the last step we have used Lemma \ref{lemma2.7}.  This contradicts
our choice of $\eta \ll \lambda^4$.

We now form a chain of consecutive dyadic cubes $\{P_i\}\subset \dd_Q$, 
connecting $Q'$ to $Q$, i.e.,
$$ 
Q'
=P_0 \lhd P_1 \lhd P_2\lhd \dots \lhd P_M \lhd P_{M+1}=Q\,,$$
where the introduced notation $P_i\lhd P_{i+1}$ means that $P_i$ is the 
dyadic child of $P_{i+1}$, that is,   $P_i\subset P_{i+1}$ and $\ell(P_{i+1})=2\ell(P_i)$.  Let $P:=P_{i_0}$, $1\le i_0\le M+1$, be the smallest of the cubes $P_i$ such that 
$Y_1 \in B_{P_i}^*$.    Setting $P':= P_{i_0-1}$, we then have that
$Y_1 \in B_P^*$, and $Y_1\notin B_{P'}^*$.   By the coherency of $\sbf$, it follows that $P\in \sbf$, so
by \eqref{eq2.2a},
\begin{equation}\label{eq3.38}
\delta(Y_1) \leq \eta\ell(P)\,.
\end{equation}  
On the other hand, 
$$\dist(Y_1,P') \gtrsim K\ell(P') \approx K \ell(P)\,,$$ and therefore, since
$y\in Q' \subset P'$,
\begin{equation}\label{eq3.39}
\delta(Y_1)\geq \lambda \ell\big(\gamma(y,Y_1)\big)
\geq \lambda|y-Y_1| \geq \lambda\dist(Y_1,P') \gtrsim
\lambda K\ell(P)\,.
\end{equation}
Combining \eqref{eq3.38} and \eqref{eq3.39}, we see that
$\lambda  \lesssim \eta/K$, which contradicts that we have fixed
$\eta \ll \lambda^4$, and $K\gg \lambda^{-4}$.
\end{proof}

\begin{lemma}\label{lemma:VQ}
Fix $\lambda\in(0,1)$. Given $Q\in\dd(\pom)$ and a non-empty set 
$V_Q\subset U_Q\cap \Omega$, such that each $ X \in V_Q$ may be 
connected by a $\lambda$-carrot path to some $y\in Q$, set
\begin{equation}\label{defi-FQ}
F_Q:= \bigcup_{X\in V_Q} F_{car}(X,Q,\lambda)\,,
\end{equation}
where  we recall that $F_{car}(X,Q,\lambda)$ is the set of $y\in Q$ that are connected via a $\lambda$-carrot path to $X$ (see \eqref{eqdefFcarQ}).  
Let $Q'\subset Q$ be such that $\ell(Q')\leq 2^{-k_0} \ell(Q)$ and $F_Q\cap Q'\neq\emptyset$. Then, there exists a non-empty set $V_{Q'}\subset U_{Q'}\cap \Omega$ such that if we 
define $F_{Q'}$ as in \eqref{defi-FQ} with $Q'$ 
replacing $Q$, then $ F_Q\cap Q'\subset  F_{Q'}$. Moreover, for 
every $Y\in V_{Q'}$, there exist $X\in V_Q$, $y\in Q'$  (indeed $y\in F_Q\cap Q'$)  and 
a $\lambda$-carrot path $\gamma=\gamma(y,X)$ such that $Y\in\gamma$.

\end{lemma}

\begin{proof}
For every $y\in F_Q\cap Q'$, by definition of $F_Q$, there 
exist $X\in V_Q$ and a $\lambda$-carrot path $\gamma=\gamma(y,X)$. 
By Lemma \ref{lemma3.15},
there is $Y=Y(y)\in\gamma\cap U_{Q'}\cap\Omega$ (there can be more than one $Y$, but we just pick one). Note that the sub-path $\gamma(y,Y)\subset \gamma(y,X)$ is also a $\lambda$-carrot path, for the same constant $\lambda$. All the conclusions in the lemma follow easily from the construction  by letting $V_{Q'}=\bigcup_{y\in F_Q\cap Q'} Y(y)$. 
\end{proof}

\begin{remark}\label{remark-VQ-children}
It follows easily from the previous proof that under the same assumptions, if one further assumes that $\ell(Q')<2^{-k_0}\,\ell(Q)$, we can then repeat the argument with both $Q'$ and $(Q')^*$ (the dyadic parent of $Q'$) to obtain respectively $V_{Q'}$ and $V_{(Q')^*}$. Moreover, this can be done in such a way that every point in $V_{Q'}$ (resp. $V_{(Q')^*}$) belongs to a $\lambda$-carrot path which also meets $V_{(Q')^*}$ (resp. $V_{Q'}$), connecting $U_Q$ and $Q'$.
\end{remark}

Given a family $\F:=\{Q_j\}\subset \dd(\pom)$ of
pairwise disjoint cubes, we recall that the ``discrete  sawtooth" $\dd_\F$ is 
the collection of all cubes in $\dd(\pom)$ that are not contained in any $Q_j\in\F$ (see \eqref{eq2.discretesawtooth1}),
and we define the restriction of $\mut$ (cf. \eqref{eq4.0}, \eqref{eq4.1}) to the sawtooth $\dd_\F$ by
\begin{equation}\label{eq4.4x}
\mut_\F(\dd'):=\mut(\dd'\cap\dd_\F)= \sum_{Q\in\dd'\setminus (\cup_{\F} \,\dd_{Q_j})}\alpha_{Q}.
\end{equation}
We then set 
$$\|\mut_\F\|_{\C(Q)}:= \sup_{Q'\subset Q}\frac{\mut_\F(\dd_{Q'})}{\sigma(Q')}.$$
Let us note that we may allow $\F$ to be empty, in which case $\dd_\F=\dd$ and $\mut_\F$ 
is simply $\mut$. We note that  the following claims  remain true when $\F$ is empty,
with some straightforward changes that are left to the interested reader. 


\begin{claim}\label{claim3.16} Given $Q\in\dd(\pom)$, and 
a family $\F=\F_Q:=\{Q_j\}\subset \dd_Q\setminus \{Q\}$ of
pairwise disjoint sub-cubes of $Q$, 
if $\|\mut_\F\|_{\C(Q)} \leq 1/2$, then each $Q' \in \dd_{\F}\cap \dd_Q$, each $Q_j\in\F$, 
and every dyadic child $Q_j'$ of any $Q_j\in \F$, belong to the 
good collection $\G$, and moreover, every such cube belongs to the {\bf same} stopping time regime
$\sbf$.  In particular, $\sbf' :=\dd_\F\cap\dd_Q$ is a semi-coherent subregime of $\sbf$, and so is
$\sbf'':= (\dd_\F\cup\F\cup\F')\cap\dd_Q$, where $\F'$ denotes the collection of all
dyadic children of cubes in $\F$.
\end{claim}
Indeed, if any $Q' \in \dd_{\F}\cap \dd_Q$ were in $\M\cup\B$ (recall that
$\M:=\{Q(\sbf)\}_{\sbf}$ is the collection of 
cubes which are the maximal elements of the stopping time regimes in $\G$), then
by construction $\alpha_{Q'}=\sigma(Q')$ for that cube (see \eqref{eq4.0}), so by definition of
$\mut$ and $\mut_\F$, we would have 
$$1=\frac{\sigma(Q')}{\sigma(Q')} \leq \frac{\mut_\F(\dd_{Q'})}{\sigma(Q')} \leq 
\|\mut_\F\|_{\C(Q)}\leq \frac12\,,$$
a contradiction.  Similarly, if some $Q_j\in\F$ (respectively, $Q_j'\in \F'$)
were in $\M\cup\B$, then its dyadic parent (respectively, dyadic grandparent)
$Q_j^*$ would belong to $\dd_{\F}\cap \dd_Q$, and by definition $\alpha_{Q_j^*}=\sigma(Q_j^*)$, so again 
we reach a contradiction.  Consequently, $\F \cup \F'\cup
(\dd_\F\cap\dd_Q)$ does not meet $\M\cup\B$, and the claim follows.

For future reference, we now prove the following.  Recall that for $Q\in \G$, $U_Q$ has precisely two
connected components $U_Q^\pm$ in $\ree\setminus \pom$.

\begin{lemma}\label{lemmaCAD} Let  $Q\in\dd(\pom)$, let $k_1$ be such that
$2^{k_1}> 2^{k_0}\gg 100 K$, 
see \eqref{eq3.15}, 
and suppose that there is 
a family $\F=\F_Q:=\{Q_j\}\subset \dd_Q\setminus \{Q\}$ of
pairwise disjoint sub-cubes of $Q$, 
with $\|\mut_\F\|_{\C(Q)} \leq 1/2$ (hence by Claim \ref{claim3.16},  there is some $\sbf$ with
$\sbf \supset (\dd_\F\cup\F\cup\F')\cap\dd_Q$), and  
 a non-empty subcollection
$\F^*\subset \F$, 
such that: 
\begin{itemize}
\item[(i)] $\ell(Q_j) \leq 2^{-k_1}\ell(Q)$, for each cube $Q_j\in \F^*$;
 \smallskip
\item[(ii)] the collection of balls $\big\{\kappa B^*_{Q_j}:= B\big(x_{Q_j},\kappa 
K \ell(Q_j)\big):\, Q_j \in \F^* \big\}$
is pairwise disjoint, where
$\kappa\gg K^4 $ is a sufficiently large positive constant; and
\smallskip
\item[(iii)] $\F^*$ has a disjoint decomposition
$\F^*=\F^*_+ \cup \F^*_-$, where for each $Q_j \in \F_\pm^*$, there is a
chord-arc subdomain $\Omega_{Q_j}^\pm\subset \Omega$, consisting 
of a union of fattened Whitney cubes $I^*$, with
$U_{Q_j}^\pm\subset \Omega_{Q_j}^\pm\subset
B^*_{Q_j}:= B(x_{Q_j},K \ell(Q_j))$, and with uniform control of the chord-arc constants. 
\end{itemize}
Define a semi-coherent subregime $\sbf^*\subset \sbf$ by
\[
	\sbf^*=\left\{Q'\in \dd_Q: Q_j\subset Q' \text{ for some }Q_j\in\F^*\right\}\,,
	\] 
and for each choice of $\pm$ for which $\F^*_\pm$ is non-empty, set
\begin{equation}\label{eqomdef}  \Omega_Q^\pm:= \Omega_{\sbf^*}^\pm \bigcup
 \left(\bigcup_{Q_j \,\in\, \F^*_\pm}\Omega^\pm_{Q_j}\right)
 \end{equation}	
	Then for $\kappa$ large enough, depending only on allowable parameters, 
$\om_Q^\pm$ is a chord-arc domain, with chord arc constants depending only on the 
 uniformly controlled chord-arc constants of $\Omega_{Q_j}^\pm$ and on the other allowable parameters.
 Moreover, $ \Omega_Q^\pm\subset  B_Q^*\cap\Omega=
 B(x_Q,K\ell(Q))\cap\Omega$, 
 and $\Omega_Q^\pm$ is a union of fattened Whitney cubes. 
\end{lemma}

\begin{remark} Note that we define $\om_Q^\pm$ if and only if $\F^*_\pm$ is non-empty.
It may be that one  of $\F^*_+, \F^*_-$ is empty, 
but $\F^*_+$ and $\F^*_-$ cannot both be empty, since $\F^*$ is non-empty by assumption.
\end{remark}

\begin{proof}[Proof of Lemma \ref{lemmaCAD}] 
Without loss of generality we may assume that $\Omega_{Q_j}\pm$ is not contained in  $\om_{\sbf^*}^\pm$ for 	
all $Q_j\in\F^*$ (otherwise we can drop those cubes from $\F^*$). On the other hand, we notice that $\Omega_Q^\pm$ 
is a union of (open) fattened Whitney cubes (assuming that it is non-empty):  each $\om_{Q_j}^\pm$ has this 
property by assumption, as does $\om_{\sbf^*}^\pm$  by construction.

We next observe that if $\Omega_Q^+$ (resp. $\Omega_Q^-$) is non-empty, then it is contained in
$\Omega$.   Indeed, by construction, $\Omega_Q^+$ is non-empty if and only if $\F^*_+$ is non-empty.
In turn, $\F^*_+$ is non-empty if and only if there is some $Q_j\in \F^*$ such that
$U_{Q_j}^+\subset \Omega_{Q_j}^+ \subset \Omega$, and moreover, the latter is true for every
$Q_j \in \F^*_+$, by definition.  But each such $Q_j$ belongs to $\sbf^*$, hence
$U_{Q_j}^+\subset  \om^+_{\sbf^*}$, again by construction (see \eqref{eq3.2}).  Thus, 
$\om_{\sbf^*}^+$ meets $\om$, and since $\om_{\sbf^*}^+\subset \ree\setminus \pom$, therefore
$\om_{\sbf^*}^+\subset \om$. Combining these observations, we see that $\Omega_Q^+\subset \om$.
Of course, the same reasoning applies to $\Omega_Q^-$, provided it is non-empty.  

In addition, 
since $\sbf^*\subset \sbf$, 
and since $K \gg K^{1/2}$, 
by Lemma \ref{lemma2.7} we have $\Omega_{\sbf^*}^\pm \subset 
B_Q^*= B(x_Q,K\ell(Q))$.  Furthermore,
$\Omega_{Q_j}^\pm\subset B^*_{Q_j}:= B(x_{Q_j},K \ell(Q_j))$, and since 
$\ell(Q_j) \leq 2^{-k_1}\ell(Q)\leq (100K)^{-1}\ell(Q)$, we obtain
\[\dist(\Omega_{Q_j}^\pm,Q) +\diam (\Omega_{Q_j}^\pm) \leq 3 K \ell(Q_j) \leq 3K 2^{-k_1} \ell(Q) \ll \ell(Q)\,.\]
Thus, in particular, $\Omega_{Q_j}^\pm\subset B_Q^*$, and therefore also 
$\om_Q^\pm\subset B_Q^*$.

It therefore remains to establish the chord-arc properties.
It is straightforward to prove the interior Corkscrew condition and the upper ADR bound, and we omit the
details.  Thus, we must verify the Harnack Chain condition,
 the lower ADR bound, and the exterior Corkscrew condition.

\smallskip
\noindent{\em Harnack Chains}.   
Suppose, without loss of generality, 
that $\Omega_Q^+$ is non-empty, and let $X,Y\in \Omega_Q^+$, with $|X-Y|= R$.
If $X$ and $Y$ both lie in $\Omega_{\sbf^*}^+$, or in the same $\om_{Q_j}^+$, then we can
connect $X$ and $Y$ by a suitable Harnack path, since each of these domains is chord-arc.
Thus, we may suppose either that 1) $X\in \Omega_{\sbf^*}^+$ and $Y$ lies in some $\om_{Q_j}^+$,
or that 2) $X$ and  $Y$  lie in two distinct $\om_{Q_{j_1}}^+$ and $\om_{Q_{j_2}}^+$.
We may reduce the latter case to the former case: 
by the separation property (ii) in Lemma \ref{lemmaCAD}, we must have
$R\gtrsim \kappa \max\big(\diam(\om_{Q_{j_1}}^+),\diam(\om_{Q_{j_2}}^+)\big)$, so given case 1), we
can connect $X\in \om_{Q_{j_1}}^+$ to the  center $Z_1$ of some $I_1^*\subset U^+_{Q_1}$,
and
$Y\in \om_{Q_{j_2}}^+$ to the center $Z_2$ of some $I_2\subset U^+_{Q_2}$, 
where $Q_1,Q_2 \in \sbf^*$,
with $Q_{j_i}\subset Q_i\subset Q$, and
$\ell(Q_i)\approx R$, $i=1,2$.  Finally, we can connect $Z_1$ and $Z_2$ using that $\Omega_{\sbf^*}^+$ is chord-arc. 

Hence, we need only construct a suitable Harnack Chain in Case 1).
We note that by assumption and construction,
$U_{Q_j}^+ \subset \Omega_{\sbf^*}^+ \cap \om_{Q_j}^+$.  

Suppose first that 
\begin{equation}\label{eqxyclose}
|X-Y|=R \leq c' \ell(Q_j)\,,
\end{equation}
where $c'\leq 1$ is a sufficiently small positive
constant to be chosen.  Since $Y \in \om^+_{Q_j} \subset B_{Q_j}^*$, we then have that
$X \in 2 B_{Q_j}^*$, so by the construction of $\Omega_{\sbf^*}^+$ and the separation
property (ii), it follows that 
$\delta(X) \geq c \ell(Q_j)$, where $c$ is a uniform constant depending only on the allowable parameters
(in particular, this fact is true for all $X \in \om_{\sbf^*}^+ \cap 2 B_{Q_j}^*$, so it does not depend on the
choice of $c'<1$).  Now choosing $c'\leq c/2$ (eventually, it may be even smaller), 
we find that $\delta(Y) \geq (c/2)\ell(Q_j)$.  
Moreover, $Y \in \om_{Q_j}^+ \subset B_{Q_j}^*$ implies that $\delta(Y) \le K\ell(Q_j)$. Also, since $X\in 2B_{Q_j}^*$ we have that $\delta(X) \le 2K\ell(Q_j)$.
Since
$\om_{Q_j}^+$ and $\om_{\sbf^*}^+$ are each the interior of a union of fattened Whitney cubes,
it follows that there are Whitney cubes $I$ and $J$, with $X\in I^*$, $Y\in J^*$, and
\[
\ell(I)\approx \ell(J) \approx \ell(Q_j)\,,\]
where the implicit constants depend on $K$.   
For  $c'$ small enough in \eqref{eqxyclose}, depending on the implicit
constants in the last display, and on the parameter $\tau$ in \eqref{whitney1},
this can happen only if $I^*$ and $J^*$ overlap (recall that we have fixed $\tau$ small enough that
$I^*$ and $J^*$ overlap if and only if $I$ and $J$ have a boundary point in common), in which case we may trivially connect
$X$ and $Y$ by a suitable Harnack Chain.

On the other hand, suppose that 
\[
|X-Y|=R \geq c' \ell(Q_j)\,.
\]
Let $Z \in  U_{Q_j}^+ \subset \Omega_{\sbf^*}^+ \cap \om_{Q_j}^+$, with
$\dist(Z, \pom_Q^+) \gtrsim \ell(Q_j)$ (we may find such a $Z$, since $U_{Q_j}^+$ is a union of fattened
Whitney cubes, all of length $\ell(I^*) \approx \ell(Q_j)$; just take $Z$ to be the center of
such an $I^*$).
We may then construct an appropriate Harnack Chain from $Y$ to $X$ by connecting
$Y$ to $Z$ via a Harnack Chain in the chord-arc domain $\Omega_{Q_j}^+$, and 
$Z$ to $X$ via a Harnack Chain in the chord-arc domain $\Omega_{\sbf^*}^+$.

\smallskip
\noindent{\em Lower ADR and Exterior Corkscrews}.  We will 
establish these two properties essentially simultaneously.  
Again suppose that, e.g.,  $\Omega_Q^+$ is non-empty. 
Let $x \in \pom_Q^+$, and consider $B(x,r)$, with 
$r < \diam \om_Q^+ \approx_K\ell(Q)$. 
Our main goal at this stage is to prove the following:
\begin{equation}\label{eqvolumebound}
\big| B(x,r) \setminus \overline{\om_Q^+}\big| \geq c r^{n+1}\,,
\end{equation}
with $c$ a uniform positive constant depending only upon allowable parameters (including $\kappa$).
Indeed, momentarily taking this estimate for granted, we may combine \eqref{eqvolumebound}
with the interior Corkscrew condition to deduce the lower ADR bound via the
relative isoperimetric inequality \cite[p. 190]{EG}.  In turn, with both the lower and upper 
ADR bounds in hand, \eqref{eqvolumebound} implies the existence of exterior Corkscrews 
(see, e.g., \cite[Lemma 5.7]{HM-I}).  

Thus, it is enough to prove \eqref{eqvolumebound}.
We consider the following cases.

\smallskip

\noindent{\bf Case 1}:  $B(x,r/2)$ does not meet $\pom_{Q_j}^+$ for any $Q_j \in \F^*_+$.
In this case, the exterior Corkscrew for $\om_{\sbf^*}^+$ associated with $B(x,r/2)$ easily implies 
\eqref{eqvolumebound}.

\smallskip

\noindent{\bf Case 2}: $B(x,r/2)$  meets $\pom_{Q_j}^+$ for at least one $Q_j \in \F^*_+$,
and $r \leq \kappa^{1/2} \ell(Q_{j_0})$, where $Q_{j_0}$ is chosen to have the largest
length $\ell(Q_{j_0})$ among those $Q_j$ such that 
$\pom_{Q_j}^+$  meets $B(x,r/2)$.  We now further split the present case into subcases.

\smallskip

\noindent{\bf Subcase 2a}:  $B(x,r/2)$  meets $\pom^+_{Q_{j_0}}$ at a point $Z$ with 
$\delta(Z) \leq (M\kappa^{1/2})^{-1} \ell(Q_{j_0})$, where $M$ is a large number to be chosen.
Then $B(Z, (M\kappa^{1/2})^{-1} r) \subset B(x,r)$, for $M$ large enough.  In addition, we claim that
$B(Z, (M\kappa^{1/2})^{-1} r)$ misses $\om_{\sbf^*}^+\cup\big(\cup_{j\neq j_0} \om_{Q_j}^+\big)$.
The fact that $B(Z, (M\kappa^{1/2})^{-1} r)$ misses every other $\om_{Q_j}^+, j\neq j_0$,
 follows immediately from the restriction  $r \leq \kappa^{1/2} \ell(Q_{j_0})$, and 
the separation property (ii).
To see that $B(Z, (M\kappa^{1/2})^{-1} r)$ misses $\om_{\sbf^*}^+$,
note that if $|Z-Y|< (M\kappa^{1/2})^{-1} r$, then 
\[\delta(Y) \leq \delta(Z) + (M\kappa^{1/2})^{-1} r \leq \left((M\kappa^{1/2})^{-1}
+M^{-1}\right) \ell(Q_{j_0}) \ll  \ell(Q_{j_0})\,,\]
for $M$ large.  On the other hand, 
\[\delta(Y) \gtrsim \ell(Q_{j_0})\,, \qquad \forall \,Y \in \om_{\sbf^*}^+ \cap
B\big(Z, \kappa^{1/2} \ell(Q_{j_0})\big)
\,,\]  by the construction of
$\om_{\sbf^*}^+$ and the separation property (ii).  Thus, the claim follows, for a
sufficiently large (fixed) choice of $M$.  Since $B(Z, (M\kappa^{1/2})^{-1} r)$ misses
$\om_{\sbf^*}^+$ and all other $\om_{Q_j}^+$, we inherit an exterior Corkscrew point in the present 
case (depending on $M$ and $\kappa$) from the chord-arc domain $\om_{Q_{j_0}}^+$.    Again \eqref{eqvolumebound} follows.

\smallskip

\noindent{\bf Subcase 2b}: $\delta(Z) \geq  (M\kappa^{1/2})^{-1} \ell(Q_{j_0})$,
for every $Z\in B(x,r/2)\cap \pom_{Q_{j_0}}^+$ (hence $\delta(Z) \approx_{\kappa,K}
\ell (Q_{j_0})$, since $\Omega^+_{Q_{j_0}} \subset B^*_{Q_{j_0}}$).  We claim that
consequently,
$x\in \partial I^*$, for some $I$ with $\ell(I) \approx \ell(Q_{j_0}) \gtrsim r$,
such that $\text{int}\, I^* \subset \om_Q^+$.  To see this, observe that it is clear 
if $x\in \pom_{Q_{j_0}}^+$ (just take $Z=x$).  Otherwise, by the separation property (ii),
the remaining 
possibility in the present scenario 
is that $x\in\partial U_{Q'}^+\cap \partial\om_{\sbf^*}^+$, for some $Q'\in \sbf^*$ with
$Q_{j_0}\subset Q'$,  in which case $\delta(x)\approx\ell(Q')\ge \ell(Q_{j_0})$.  Since also
$\delta(x)\le |x-Z|+\delta(Z)\lesssim_{\kappa, K}\ell(Q_{j_0})$, for any
$Z\in B(x,r/2)\cap\pom^+_{Q_{j_0}}$, the claim follows.

On the other hand, since 
$x\in \pom_Q^+$, there is a $J\in\W$ with $\ell(J) \approx \ell(Q_{j_0})$, such that
$J^*$ is not contained in $\om_Q^+$.  We then have an exterior Corkscrew point in
$J^* \cap B(x,r)$, and \eqref{eqvolumebound} follows in this case.

\smallskip

\noindent{\bf Case 3}: $B(x,r/2)$  meets $\pom_{Q_j}^+$ for at least one $Q_j \in \F^*_+$,
and $r > \kappa^{1/2} \ell(Q_{j_0})$, where as above $Q_{j_0}$ has the largest
length $\ell(Q_{j_0})$ among those $Q_j$ such that 
$\pom_{Q_j}^+$  meets $B(x,r/2)$.   In particular then,
$r \gg 2K \ell(Q_{j_0})=\diam(B^*_{Q_{j_0}})\geq \diam(\om^+_{Q_{j_0}})$, since we assume
$\kappa \gg K^4$.

We next claim that $B(x,r/4)$ contains some $x_1 \in \pom_{\sbf^*}^+ \cap
\pom_Q^+$.  This is clear if $x\in \partial\om_{\sbf^*}^+$ by taking $x_1=x$. Otherwise,  $x\in\partial\Omega_{Q_j}^+$ for some $Q_j\in\F^*$. Note that $U_{Q_j}^{\pm}\subset B(x_{Q_j}, K\ell(Q_j))\subset B(x, 2 K\ell(Q_{j}))$. Also,  
$U_{Q_j}^{\pm}\subset \om_{\sbf^*}^\pm$, by construction. 
On the other hand we note that if 
$Z\in U_Q^{\pm}$ we have by \eqref{dist:UQ-pom}  
\[
|Z-x_{Q_j}|\ge \delta(Z)\gtrsim \eta^{1/2}\ell(Q) 
\ge \eta^{1/2}2^{k_1}\ell(Q_{j})
\gg K\ell(Q_{j})
\]  
by our choice of $k_1$. 
By this fact, and
the definition of $\om_{\sbf^*}$, we have
$$U_Q^{\pm}
\subset  \om_{\sbf^*}^\pm\setminus B(x, 3 K\ell(Q_{j}))\,.$$ 
Using then that $\om_{\sbf^*}^\pm$ is connected, 
we see that  a path within $\om_{\sbf^*}^\pm$ joining  
$U_{Q_j}^{\pm}$ with $U_{Q}^{\pm}$ must meet $\partial B(x,3K\ell(Q_{j}))$.
Hence we can find  $Y^\pm\in \om_{\sbf^*}^\pm\cap \partial B(x,3K\ell(Q_{j}))$. 
By Lemma \ref{lemma2.7},  $\om_{\sbf^*}^+$ and $\om_{\sbf^*}^-$ are disjoint 
(they live respectively above and below the graph $\Gamma_{\sbf}$), 
so a path joining 
$Y^+$ and $Y^-$  within   $\partial B(x,3K\ell(Q_{j}))$ meets some 
$x_1\in \partial \om_{\sbf^*}^+\cap \partial B(x,3K\ell(Q_{j}))$. 
On the other hand, 
$x_1\notin \overline{\Omega_{Q_j}^+}$, since 
$\overline{\Omega_{Q_j}^+}\subset \overline{B_{Q_j}^*}\subset 
B(x,3K\ell(Q_{j}))$. Furthermore,
$x_1\in \partial B(x,3K\ell(Q_{j}))\subset \kappa B_{Q_j}^*$, 
so by assumption (ii), we necessarily have that $x_1\notin \overline{\Omega_{Q_k}^+}$ for $k\neq j$. Thus, 
$x_1\in\partial\Omega_Q^+$, and moreover,
since  $B(x,r/2)$ meets $\partial \Omega_{Q_j}^+$ (at $x$) we have
$\ell(Q_j)\le\ell(Q_{j_0})$.   Therefore, 
$x_1$ is the claimed point, since in the current case 
$3K\ell(Q_{j})\le 3K\ell(Q_{j_0}) \ll r$.

With the point $x_1$ in hand, we note that 
\begin{equation}\label{eqballcontain}
B(x_1,r/4) \subset B(x,r/2)\quad \text{and } \quad B(x_1,r/2) \subset B(x,r)\,.
\end{equation}
By the exterior Corkscrew condition for $\om_{\sbf^*}^+$, 
\begin{equation}\label{eqcs1}
\big|B(x_1,r/4) \setminus \overline{\om_{\sbf^*}^+} \,\big| \geq \,c_1 r^{n+1}\,,
\end{equation}
for some constant $c_1$ depending only on $n$ and the ADR/UR constants for $\pom$, by 
Lemma \ref{lemma2.7}.
Also, for each $\om^+_{Q_j}$ whose boundary meets 
$B(x_1,r/4) \setminus \overline{\om_{\sbf^*}^+}$ (and thus meets $B(x,r/2)$),
\begin{equation}\label{eq3.26}
\kappa^{1/4}\diam(B^*_{Q_j}) \leq 
\kappa^{1/4}\diam(B^*_{Q_{j_0}})\leq 2K\kappa^{1/4}\ell(Q_{j_0})
 \leq \frac{2K r}{\kappa^{1/4}} \ll r\,, 
 \end{equation}
in the present scenario. 
Consequently, $\kappa^{1/4}B_{Q_j}^* \subset B(x_1,r/2)$, for all such $Q_j$.

We now make the following claim.

\noindent{\em Claim 1}:
\begin{equation}\label{eqcs2}
\big|B(x_1,r/2) \setminus \overline{\om^+_Q}
\,\big| \geq \,c_2 r^{n+1}\,,
\end{equation}
for some $c_2>0$ depending only on allowable parameters.

Observe that by the second
containment in \eqref{eqballcontain},  we obtain \eqref{eqvolumebound} as an  immediate
consequence of \eqref{eqcs2}, and thus the proof will be complete once we have
established Claim 1.

 To prove the claim, we
suppose first that
\begin{equation}\label{eqcs3}
\sum \big| B_{Q_j}^*  \setminus \overline{\om_{\sbf^*}^+} \,\big| \leq\, \frac{c_1}{2}\, r^{n+1}\,,
\end{equation}
where the sum runs over those $j$ such that
$\overline{B_{Q_j}^*}$ meets $B(x_1,r/4)\setminus   \overline{\om_{\sbf^*}^+}$, 
and $c_1$ is the constant in \eqref{eqcs1}.  In that case, \eqref{eqcs2} holds with
$c_2 = c_1/2$  (and even with
$B(x_1,r/4)$), by definition of $\om^+_Q$  (see \eqref{eqomdef}), and the fact that
$\om_{Q_j} \subset B_{Q_j}^*$.  
On the other hand, if
\eqref{eqcs3} fails, then summing over the same subset of indices $j$, we have
\begin{equation}\label{eqcs4}
C K \sum \ell(Q_j)^{n+1}   \geq \sum
\big| B_{Q_j}^*  \setminus \overline{\om_{\sbf^*}^+} \,\big| \geq\, \frac{c_1}{2}\, r^{n+1}
\end{equation}
We now make a second claim:

\noindent{\em Claim 2}:  For $j$ appearing in the previous sum, we have 
\begin{equation}\label{eqcs5}
 \big| \left(\kappa^{1/4}B_{Q_j}^*\setminus B_{Q_j}^*\right)  
 \setminus \overline{\om_{\sbf^*}^+} \,\big| \geq\, c\, \ell(Q_j)^{n+1}\,,
\end{equation}
for some uniform $c>0$.  

Taking the latter claim for granted momentarily, we insert 
estimate \eqref{eqcs5} into \eqref{eqcs4} and sum, to obtain
\begin{equation}\label{eqcs6}
\sum  \big| \left(\kappa^{1/4}B_{Q_j}^*\setminus B_{Q_j}^*\right)  
 \setminus \overline{\om_{\sbf^*}^+} \,\big|  \gtrsim r^{n+1}\,.
\end{equation}
By the separation property (ii), the balls $\kappa^{1/4} B^*_{Q_{j}}$ are pairwise disjoint, and
by assumption $\om_{Q_j}^+\subset B_{Q_j}^*$.  Thus,
for any given $j_1$,  $\kappa^{1/4} B^*_{Q_{j_1}} \setminus \overline{B^*_{Q_{j_1}}}$ misses 
$\cup_j \overline{\Omega^+_{Q_j}}$.   Moreover, as noted above (see \eqref{eq3.26}
and the ensuing comment), $\kappa^{1/4}B_{Q_j}^* \subset B(x_1,r/2)$ for each $j$
under consideration in \eqref{eqcs3}-\eqref{eqcs6}.  Claim 1 now follows.

We turn to the proof of Claim 2.  There are two cases:  if $\frac12 \kappa^{1/4} B_{Q_j}^* \subset
\ree\setminus \overline{\om_{\sbf^*}^+}$, then \eqref{eqcs5} is trivial, since $\kappa \gg 1$.  Otherwise,
$\frac12 \kappa^{1/4} B_{Q_j}^*$ contains a point $z\in \pom_{\sbf^*}^+$.  In the latter case,
by the exterior Corkscrew condition for $\om_{\sbf^*}^+$, 
\[\big| B\big(z, 2^{-1}\kappa^{1/4}K \ell(Q_j)\big) \setminus \overline{\om_{\sbf^*}^+}\,\big|
\,\gtrsim \,\kappa^{(n+1)/4}\big(K\ell(Q_j)\big)^{n+1}  \gg \, |B_{Q_j}^*|\,,
\]
since $\kappa \gg 1$.   On the other hand, $ B\big(z, 2^{-1}\kappa^{1/4}K \ell(Q_j)\big)
\subset \kappa^{1/4} B_{Q_j}^*$, and \eqref{eqcs5} follows.
\end{proof}
\subsection{Step 2:  Proof of $H[M_0,1]$}\label{ss3.2}  We shall deduce $H[M_0,1]$ from the following pair of claims.

\begin{claim}\label{claim3.17} $H[0,\theta]$ holds for every $\theta\in(0,1]$.
\end{claim}
\begin{proof}[Proof of Claim \ref{claim3.17}] 
	If $a=0$ in \eqref{eq3.12}, then $\|\mut\|_{\C(Q)}=0$, whence it follows by Claim \ref{claim3.16}, with
	$\F=\emptyset$, that there is a stopping time regime $\sbf\subset \G$, with $\dd_Q\subset \sbf$.  
	Hence $\sbf':=\dd_Q$ is a coherent subregime of $\sbf$,  so by 
	Lemma \ref{lemma2.7}, 
	each of $\om_{\sbf'}^\pm$ 
	is a CAD, containing $U_Q^\pm$, respectively, with $\om_{\sbf'}^\pm\subset B_Q^*$ by \eqref{def:BQ*}.  
	 Moreover,  by \cite[Proposition A.14]{HMM}
$$
Q \subset
\pom_{\sbf'}^\pm \cap\pom\,,
$$ 
so that $\sigma(Q)\leq \sigma(\pom_{\sbf'}^\pm \cap\pom).$ 
	Thus, $H[0,\theta]$ holds trivially.
\end{proof}

\begin{claim}\label{claim3.18}
	There is a uniform constant $b>0$ such that
	$H[a,1] \implies H[a+b,1]$, for all $a\in [0,M_0)$.
\end{claim}
Combining Claims \ref{claim3.17} and \ref{claim3.18}, we find that $H[M_0,1]$ holds.

To prove Claim \ref{claim3.18}, we shall require the following.
\begin{lemma}[{\cite[Lemma 7.2]{HM-I}}]\label{lemma:Corona}
	Suppose that $E$ is an $n$-dimensional ADR set,
	and let $\mut$ be a discrete Carleson measure, as in \eqref{eq4.1}-\eqref{eq4.7a}
	above.   Fix $Q\in \dd(E)$.
	Let $a\geq 0$ and $b>0$, and suppose that
	$\mut(\dd_{Q})\leq (a+b)\,\sigma(Q).$
	Then there is a family $\F=\{Q_j\}\subset\dd_{Q}$
	of pairwise disjoint cubes, and a constant $C$ depending only on $n$
	and the ADR constant such that
	\begin{equation} \label{Corona-sawtooth}
	\|\mut_\F\|_{\C(Q)}
	\leq C b,
	\end{equation}
	\begin{equation}
	\label{Corona-bad-cubes}
	\sigma \bigg(\bigcup_{\F_{bad}}Q_j\bigg)
	\leq \frac{a+b}{a+2b}\, \sigma(Q)\,,
	\end{equation}
	where $\F_{bad}:=
	\{Q_j\in\F:\,\mut\big(\dd_{Q_j}\setminus \{Q_j\}\big)>\,a\sigma(Q_j)\}$.
\end{lemma}

We refer the reader to \cite[Lemma 7.2]{HM-I} for the proof.   We remark that the lemma is stated
in \cite{HM-I} in the case that $E$ is the boundary of a connected domain, but the proof actually requires only that
$E$ have a dyadic cube structure, and that $\sigma$ be a
non-negative, dyadically doubling Borel measure on $E$.  In our case, we shall of course apply the lemma with
$E=\pom$, where $\om$ is open, but not necessarily connected.

\begin{proof}[Proof of Claim \ref{claim3.18}]  We assume that $H[a,1]$ holds, for some  $a\in[0,M_0)$.
	Set $b= 1/(2C)$, where $C$ is the constant in
	\eqref{Corona-sawtooth}.
	Consider a cube $Q\in \dd(\pom)$ with
	$\mut(\dd_Q) \leq (a+b) \sigma(Q)$.   Suppose that there is a set $V_Q\subset U_Q\cap\om$ such that
	\eqref{eq3.10} holds with $\theta=1$.  We fix $k_1>k_0$ (see \eqref{eq3.15}) large 
	enough so that $2^{k_1}> 100 K$. 
	
	\smallskip
	
	\noindent{\bf Case 1}: There exists $Q'\in\dd_{Q}^{k_1}$ (see \eqref{eq3.4aaa})
	with  $\mut(\dd_{Q'}) \le a\sigma(Q')$.  
	
	In the present scenario $\theta=1$, that is, $\sigma(F_Q)=\sigma(Q)$ (see 
	\eqref{eq3.10} and \eqref{defi-FQ}), which implies $\sigma(F_Q\cap Q')=\sigma(Q')$. We apply Lemma \ref{lemma:VQ} to obtain $V_{Q'}\subset U_{Q'}\cap\Omega$ and the corresponding $F_{Q'}$ which satisfies $\sigma(F_{Q'})=\sigma(Q')$. That is, \eqref{eq3.10} holds for $Q'$, with $\theta=1$.  Consequently, we may 
	apply the induction hypothesis
	to $Q'$, to find $V^*_{Q'}\subset V_{Q'}$, such 
	that for each $U_{Q'}^i$ meeting $V^*_{Q'}$, there is a chord-arc domain 
	$\Omega_{Q'}^i \supset U_{Q'}^i$ formed by a union of fattened Whitney cubes with $\Omega^i_{Q'}\subset B(x_Q',K\ell(Q'))\cap\Omega$, and 
\begin{equation}\label{eq3.34}
\sum_{i:  U_{Q'}^i \text{ meets } V^*_{Q'}}\sigma (\pom^i_{Q'}\cap Q')\geq  c_a\sigma(Q')\,.
\end{equation}
	By  Lemma \ref{lemma:VQ}, and since $k_1>k_0$, each $Y\in V^*_{Q'}$ lies 
	on a $\lambda$-carrot path connecting some $y\in Q'$ to some $X\in V_Q$; let $V^{**}_Q$ 
	denote the 
	set of all such $X$, and let ${\bf U}^{**}_Q$ (respectively, $\sub_{Q'}$) denote the collection of
	connected components of $U_Q$ (resp., of
	$U_{Q'}$) which meet $V^{**}_Q$ (resp., $V^*_{Q'}$).  
	By construction, each component $U_{Q'}^i \in \sub_{Q'}$
	may be joined to some corresponding component in ${\bf U}^{**}_Q$,
	via one of the 
	carrot paths.
	After possible
	renumbering, we designate this component as $U_Q^i$,  we let
	$X_i, Y_i$ denote the points in $V_Q^{**}\cap U_Q^i$
	and in $V_Q^*\cap U_{Q'}^i$, respectively, that are joined by this carrot path,
	and we let $\gamma_i$ be the portion of the carrot path joining $X_i$ to $Y_i$
	(if there is more than one such path or component, we just pick one).  We also let
	$V_Q^*=\{X_i\}_i$ be the collection of all of the selected points $X_i$.  We let 
	$\W_i$ be the collection of Whitney cubes meeting
	$\gamma_i$, and we then define 
$$\Omega_Q^i:= \Omega^i_{Q'} \bigcup \interior\left(\bigcup_{I\in\W_i}\,I^*\right) \bigcup
U_{Q}^i\,.$$
By the definition of a $\lambda$-carrot path, since 
$\ell(Q')\approx_{k_1}\ell(Q)$, and since  $ \Omega^i_{Q'}$ is a CAD,
one may readily verify that $\Omega^i_Q$ is also a CAD consisting of a union $\cup_k I_k^*$ 
of fattened Whitney cubes $I_k^*$.
We omit the details.  Moreover, by construction, 
$$\pom_Q^i\cap Q \supset \pom_{Q'}^i\cap Q',$$
so that the analogue of \eqref{eq3.34} holds with $Q'$ replaced by $Q$, and with $c_a$ replaced by 
$c_{k_1} c_a$.

It remains to verify that $\Omega_Q^i \subset B_Q^* = B(x_Q, K \ell(Q))$.
	By the induction hypothesis, and our choice of $k_1$,  since $\ell(Q')=2^{-k_1}\ell(Q)$ we have
	\[
	\Omega^i_{Q'}\subset B_{Q'}^*\cap\Omega = B(x_{Q'},K\ell(Q'))\cap\Omega
\subset B_Q^*\cap\Omega.
	\]
Moreover, $U_Q\subset B_Q^*$, by \eqref{def:BQ*}.   We therefore need only to consider $I^*$ with $I\in\W_i$.
For such an $I$, by definition there is a point $Z_i\in I\cap\gamma_i$ and $y_i\in Q'$, so that
	$Z_i\in\gamma(y_i,X_i)$ and thus, 
	\[
	\delta(Z_i)
	\le
	|Z_i-y_i|
	\le
	\ell(y_i,Z_i)
	\le
	\ell(y_i,X_i)
	\le
	\lambda^{-1}\delta(X_i)
	\le
	\lambda^{-1}|X_i-x_Q|
	\le
	\lambda^{-1} CK^{1/2}\ell(Q)\,,
	\]
where in the last inequality we have used \eqref{dist:UQ-pom} and the fact that $X_i\in U_Q$.	Hence, for every $Z\in I^*$ by \eqref{Whintey-4I}
	\[
	|Z-x_Q|
	\le
	\diam(2I)+|Z_i-y_i|+|y_i-x_Q|
	\le
	C|Z_i-y_i|+\diam(Q)
	<K\ell(Q),
	\]
	by our choice of the parameters $K$ and $\lambda$. 

We then obtain the conclusion of $H[a+b,1]$ in the present case.
	
	\smallskip

	\noindent{\bf Case 2}:  $\mut(\dd_{Q'}) > a\sigma(Q')$ for every $Q'\in\dd_{Q}^{k_1}$.

	In this case, we apply Lemma \ref{lemma:Corona}
	to obtain a pairwise disjoint family $\F=\{Q_j\}\subset \dd_Q$ such 
	that \eqref{Corona-sawtooth} and \eqref{Corona-bad-cubes} hold.	
	In particular, by our choice of $b=1/(2C)$,
	\begin{equation}\label{eq3.27bb}
	\|\mut_\F\|_{\C(Q)}\leq 1/2\,,
	\end{equation}
	so that the conclusions of Claim \ref{claim3.16} hold. 
	
	We set 
	\begin{equation}\label{eq3.27aa}
	F_0:= Q\setminus \bigg(\bigcup_\F Q_j\bigg)\,,
	\end{equation}
	define
	\begin{equation}\label{eq3.27a}\F_{good}:= \F\setminus \F_{bad}=
	\big\{Q_j\in\F:\,\mut\big(\dd_{Q_j}\setminus \{Q_j\}\big)\leq\,a\sigma(Q_j)\big\}\,,
	\end{equation}
	and let
	$$G_0:= \bigcup_{\F_{good}}Q_j\,.$$
	Then by \eqref{Corona-bad-cubes}
	\begin{equation}\label{eq3.25}
	\sigma(F_0 \cup G_0 )\, \geq \,\rho\sigma(Q)\,,
	\end{equation}
	where $\rho\in (0,1)$ is defined by
	\begin{equation}\label{eq3.42aa} \frac{a+b}{a+2b} \leq \frac{M_0+b}{M_0+2b}=:1-\rho \in (0,1)\,.
	\end{equation}
	We claim that 
	\begin{equation}\label{eq:erwgte}
	\ell(Q_j)\le 2^{-k_1}\,\ell(Q),
	\qquad
	\forall\,Q_j\in\F_{good}.
	\end{equation}
	Indeed, if this were not true for some $Q_j$, 
	then by definition of $\F_{good}$ and pigeon-holing there 
	will be $Q_j'\in\dd_{Q_j}$ with $\ell(Q_j')=2^{-k_1}\,\ell(Q)$ such 
	that $\mut(\dd_{Q_j'})\le a\,\sigma(Q_j')$. This contradicts the assumptions of the current case. 
	
	Note also that $Q\notin\F_{good}$ by \eqref{eq:erwgte} and $Q\notin\F_{bad}$ by \eqref{Corona-bad-cubes}, hence  $\F\subset\dd_Q\setminus\{Q\}$. By \eqref{eq3.27bb} and Claim \ref{claim3.16}, 
	there is some stopping time regime $\sbf\subset\G$ so that 
	$\sbf''=(\dd_\F\cup\F\cup\F')\cap\dd_Q$  
	is a semi-coherent subregime of $\sbf$,  where $\F'$ denotes the collection of all
	dyadic children of cubes in $\F$.
	
	\smallskip
	
	\noindent{\bf Case 2a}:   $\sigma(F_0) \geq \frac12\,  \rho \sigma(Q)$.
	
	In this case,
	$Q$ has an ample overlap with the boundary of a chord-arc domain with controlled chord-arc constants. Indeed, 
	let $\sbf'=\dd_\F\cap\dd_Q$ which, by \eqref{eq3.27bb} and Claim \ref{claim3.16}, is a semi-coherent subregime of some $\sbf\subset\G$. Hence, by Lemma \ref{lemma2.7},
	each of $\om_{\sbf'}^\pm$ is a CAD
with constants depending on the allowable parameters, formed by the union of fattened Whitney boxes, 
which satisfies
$\om_{\sbf'}^\pm\subset B_Q^*\cap\Omega$  (see \eqref{eq3.3aa}, \eqref{eq3.2}, and \eqref{def:BQ*}). Moreover, by 
\cite[Proposition A.14]{HMM} and \cite[Proposition 6.3]{HM-I} and our current assumptions,
	\[ 
	\sigma(Q\cap \pom_{\sbf'}^\pm)= \sigma(F_0)\geq \frac{\rho}{2} \sigma(Q)\,.
	\]
	Recall that in establishing $H[a+b,1]$, we assume that 
	there is a set $V_Q\subset U_Q\cap\om$ for which
	\eqref{eq3.10} holds with $\theta=1$. Pick then $X\in V_Q$ and set $V_Q^*:=\{X\}\subset V_Q$. Note that since $U_Q=U_Q^+\cup U_Q^-$ it follows that $X$ belongs to either $U_Q^+\cap\Omega$ or $U_Q^-\cap\Omega$. For the sake of specificity assume that $X\in U_Q^+\cap\Omega$ hence, in particular, 
	$U_Q^+\subset\om_{\sbf'}^+\subset \om$.  Note also that $U_Q^+$ is the only component of $U_Q$ meeting $V_Q^*$. All these together give at once that the conclusion of $H[a+b,1]$ holds in the present case.
	
	\smallskip
	
	\noindent{\bf Case 2b}:   $\sigma(F_0) < \frac12 \,\rho \sigma(Q)$.
	
	In this case  by \eqref{eq3.25} 
	\begin{equation}\label{eq3.33}
	\sigma(G_0)\,\geq \,\frac{\rho}{2}\, 
	\sigma(Q)\,.
	\end{equation}
	In addition, by the definition of $\F_{good}$ \eqref{eq3.27a}, and pigeon-holing, every $Q_j\in\F_{good}$ has a dyadic
	child $Q_j'$ (there could be more children satisfying this, but we just pick one) so that 
	\begin{equation}\label{eq3.27}
	\mut(\dd_{Q'_j}) \leq a\sigma(Q'_j)\,.
	\end{equation}
	Under the present assumptions $\theta=1$, that is, $\sigma(F_Q)=\sigma(Q)$ (see
	\eqref{eq3.10} and \eqref{defi-FQ}), 
	hence $\sigma(F_Q\cap Q_j')=\sigma(Q_j')$. 
	We apply Lemma \ref{lemma:VQ} (recall \eqref{eq:erwgte}) to obtain  $V_{Q_j'}\subset U_{Q_j'}\cap \Omega$ and 
	$F_{Q_j'}$ which satisfies $\sigma(F_{Q_j'})=\sigma(Q_j')$. That is, \eqref{eq3.10} holds 
	for $Q_j'$, with $\theta=1$. 
	Consequently, recalling that 
	$Q_j'\in\sbf\subset\G$ (see Claim \ref{claim3.16}),
	and applying the induction hypothesis
	to $Q_j'$, we find $V^*_{Q_j'}\subset V_{Q_j'}$, such 
	that for each $U_{Q_j'}^\pm$ meeting $V^*_{Q_j'}$, there is a chord-arc domain 
	$\Omega_{Q_j'}^\pm \supset U_{Q_j'}^\pm$ formed by a union of fattened Whitney cubes with $\Omega^\pm_{Q_j'}\subset B_{Q_j'}^*\cap\Omega$. 
	Moreover, since in particular, the cubes in $\F$ along with all of their children
	belong to the same stopping time regime $\sbf$ (see Claim \ref{claim3.16}),
	 the connected component $U_{Q_j}^\pm$ overlaps with
the corresponding component $U^\pm_{Q_j'}$ for its child,  so we may augment
$\Omega_{Q_j'}^\pm$ by adjoining to it the appropriate component $U_{Q_j}^\pm$, to form
a chord arc domain 
\begin{equation}\label{cadbuild}
\Omega^\pm_{Q_j} := \Omega_{Q_j'}^\pm \cup U_{Q_j}^\pm\,.
\end{equation}
Moreover, since  $K\gg 1$, and since $Q_j'\subset Q_j$, we have that
$B^*_{Q_j'} \subset B^*_{Q_j}$, hence $\Omega^\pm_{Q_j}\subset B_{Q_j}^*$ by construction.  

	
By a covering lemma argument, for a sufficiently large constant $\kappa\gg K^4$, 
	we may extract a 
	subcollection $\F_{good}^*\subset \F_{good}$ so that 
	$ \{ \kappa B_{Q_j}^*\}_{Q_j\in \F_{good}^*}$ is a pairwise disjoint family, and 
	\[
	\bigcup_{ Q_j\in \F_{good}} Q_j \subset\bigcup_{ Q_j\in \F_{good}^*} 5\kappa B_{Q_j}^*.
	\]
	In particular, by \eqref{eq3.33},
	\begin{equation}\label{G0*}
\sum_{ Q_j\in \F_{good}^*}\sigma( Q_j)
	\gtrsim_{\kappa,K}
\sum_{ Q_j\in \F_{good}}\sigma( Q_j) =\sigma(G_0)
	\gtrsim
\rho\sigma(Q),
	\end{equation}
	where the implicit constants depend  on ADR, $K$, and  
	the dilation factor $\kappa$. 



By the induction hypothesis, and by construction \eqref{cadbuild} and ADR,
\begin{equation}\label{eqx}
\sigma (Q_j\cap \pom_{Q_j})  \gtrsim \sigma(Q'_j)\gtrsim \sigma(Q_j)\,,
\end{equation}
where  $\Omega_{Q_j} $ is equal either to $\Omega_{Q_j}^+$ or to $\Omega_{Q_j}^-$
(if \eqref{eqx} holds for both choices, we arbitrarily set $\Omega_{Q_j}=\Omega_{Q_j}^+$).

Combining \eqref{eqx} with \eqref{G0*},
we obtain
\begin{equation}\label{eqxx}
\sum_{Q_j \in \F^*_{good}} \sigma (Q_j\cap \pom_{Q_j})  \gtrsim \sigma(Q)\,.
\end{equation}

We now assign each $Q_j\in \F_{good}^*$ either to
$\F^*_+$ or to $\F^*_-$, depending on whether we chose $\Omega_{Q_j}$ satisfying
\eqref{eqx} to be $\Omega_{Q_j}^+$, or $\Omega_{Q_j}^-$.  
We note that at least one of the sub-collections $\F^*_\pm$ 
is non-empty, since
for each $j$, there was at least one choice of ``+' or ``-" such that
\eqref{eqx} holds for the corresponding choice of $\Omega_{Q_j}$.  Moreover, the 
two collections are disjoint, since we have arbitrarily designated $\Omega_{Q_j} = \Omega_{Q_j}^+$
in the case that there were two choices for a particular $Q_j$.


	To proceed, as in Lemma \ref{lemmaCAD} we set 
	\[
	\sbf^*=\left\{Q'\in \dd_Q: Q_j\subset Q' \text{ for some }Q_j\in\F_{good}^*\right\}
	\] 
which is semi-coherent by construction. 
For $\F^*_\pm$ non-empty, we now define
\begin{equation}\label{defi:Omega-Q-pm}
	\Omega_{Q}^{\pm}
	=
	\Omega_{\sbf^*}^{\pm}\bigcup 
	\Big(\bigcup_{Q_j\in \F_{\pm}^{*}} \Omega_{Q_j}\Big).
	\end{equation}
Observe that by the induction hypothesis, 
 and our construction (see \eqref{cadbuild} and the ensuing comment),
 for an appropriate choice of $\pm$,
 $U_{Q_j}^\pm\subset \Omega_{Q_j} \subset B_{Q_j}^*$, and since $\ell(Q_j) \leq 2^{-k_1} \ell(Q)$,
by \eqref{eqxx} and Lemma  \ref{lemmaCAD},  with $\F^*=\F^*_{good}$,
each (non-empty) choice of $\Omega_Q^\pm$
defines a chord-arc domain with the requisite properties.

	Thus, we have proved Claim \ref{claim3.18}
	and therefore, as noted above, it follows that $H[M_0,1]$ holds.
\end{proof}

\subsection{Step 3:  bootstrapping $\theta$} In this last step, we shall prove 
that there is a uniform constant $\zeta\in (0,1)$ such that for each $\theta\in (0,1]$,
$H[M_0,\theta]\implies H[M_0,\zeta\theta]$.  Since we have already established
$H[M_0,1]$, we then conclude that $H[M_0,\theta_1]$ holds for any given
$\theta_1\in (0,1]$.  As noted above, it then follows that Theorem
\ref{t1} holds, as desired.

In turn, it will be enough to verify the following.  
\begin{claim}\label{claim3.52}   There is a uniform constant $\beta \in(0,1)$ such that
for every $a\in [0,M_0)$, $\theta\in (0,1]$, $\vartheta\in (0,1)$, and $b=1/(2C)$ as in Step 2/Proof of Claim \ref{claim3.18}, 
if $H[M_0,\theta]$ holds, then
$$H[a,(1-\vartheta)\theta]\implies H[a+b, (1-\vartheta\beta)\theta]\,.$$
\end{claim}
Let us momentarily take Claim \ref{claim3.52} for granted.  Recall that by Claim \ref{claim3.17},
$H[0,\theta]$ holds for all $\theta\in (0,1]$.  In particular, given $\theta \in (0,1]$ fixed, for
which $H[M_0,\theta]$ holds, we have that
$H[0,\theta/2]$ holds.    
Combining the latter fact with Claim \ref{claim3.52}, and iterating,
we obtain that $H[kb, (1-2^{-1}\beta^k)\theta]$ holds.  We eventually reach
$H[M_0,(1-2^{-1}\beta^\nu)\theta]$, with $\nu \approx M_0/b$.  The conclusion of Step 3 now follows,
with $\zeta := 1-2^{-1}\beta^\nu$.

\begin{proof}[Proof of Claim \ref{claim3.52}]
The proof will be a refinement of that of Claim \ref{claim3.18}.  We are given some $\theta\in(0,1]$ such that
$H[M_0,\theta]$ holds, and
we assume that $H[a,(1-\vartheta)\theta]$ holds, for some  $a\in[0,M_0)$ and $\vartheta\in (0,1)$.
Set $b= 1/(2C)$, where as before $C$ is the constant in
\eqref{Corona-sawtooth}.
Consider a cube $Q\in \dd(\pom)$ with
$\mut(\dd_Q) \leq (a+b) \sigma(Q)$.   Suppose that there is a set $V_Q\subset U_Q\cap\om$ such that
\eqref{eq3.10} holds with $\theta$ replaced by
$(1-\vartheta\beta)\theta$, for some $\beta\in (0,1)$ to be determined.
Our goal is to show that for a sufficiently small, but uniform choice of $\beta$,
we may deduce the conclusion of the induction hypothesis, with $\C_{a+b}, c_{a+b}$ in place of 
$C_a,c_a$.

By assumption, and recalling the definition of $F_Q$ in \eqref{defi-FQ}, 
we have that \eqref{eq3.10} holds with constant $(1-\vartheta\beta)\theta$, i.e.,
\begin{equation}\label{eq3.56}
\sigma(F_Q) \geq (1-\vartheta\beta)\theta \sigma(Q)\,.
\end{equation}

As in Step 2, we fix $k_1>k_0$ (see \eqref{eq3.15}) large 
enough so that $2^{k_1}> 100K$. 
There are two principal cases.  The first is as follows.

\smallskip


\noindent{\bf Case 1}: There exists $Q'\in\dd_{Q}^{k_1}$ (see \eqref{eq3.4aaa}) with  
$\mut(\dd_{Q'}) \le a\sigma(Q')$.


We split Case 1 into two subcases.

\smallskip

\noindent{\bf Case 1a}:  $\sigma(F_Q\cap Q') \geq (1-\vartheta)\theta \sigma(Q')$.

In this case, we follow the Case 1 argument for $\theta=1$ in 
Subsection \ref{ss3.2} {\it mutatis mutandis}, so we merely sketch the proof.
By Lemma \ref{lemma:VQ}, we may construct $V_{Q'}$ and $F_{Q'}$ so that
$F_Q\cap Q' = F_{Q'}$ and hence $\sigma(F_{Q'})\ge (1-\vartheta)\theta \sigma(Q')$. 
We may then apply the induction hypothesis
$H[a,(1-\vartheta)\theta]$ in $Q'$, and then proceed exactly as in Case 1 of Step 2  to construct
a subset $V_Q^*\subset V_Q$ and a family of chord-arc domains $\Omega_Q^i$ satisfying the
various desired properties, and such that
$$\sum_{i:  U_{Q}^i \text{ meets } V^*_{Q}}\sigma (\pom^i_{Q}\cap Q)\geq c_a\sigma(Q')
\gtrsim_{k_1}c_a \sigma(Q)\,.$$

The conclusion of $H[a+b, (1-\vartheta\beta)\theta]$ then holds in the present scenario.

\smallskip

\noindent{\bf Case 1b}:  $\sigma(F_Q\cap Q') < (1-\vartheta)\theta \sigma(Q')$.

By \eqref{eq3.56}
$$(1-\vartheta\beta)\theta \sigma(Q)\, \leq \,\sigma(F_Q) \,=\, \sigma(F_Q\cap Q') + 
\sum_{Q''\in\dd_Q^{k_1}\setminus\{Q'\}}\sigma(F_Q\cap Q'')\,.$$
In the scenario of Case 1b, this leads to
\begin{multline*}
(1-\vartheta\beta)\theta \sigma(Q')+ (1-\vartheta\beta)\theta
\sum_{Q''\in\dd_Q^{k_1}\setminus\{Q'\}}\sigma(Q'')=
(1-\vartheta\beta)\theta \sigma(Q)\\[4pt] \leq \, (1-\vartheta)\theta \sigma(Q') + 
\sum_{Q''\in\dd_Q^{k_1}\setminus\{Q'\}}\sigma(F_Q\cap Q'')\,,
\end{multline*}
that is,
\begin{equation}\label{eq3.60}
(1-\beta)\vartheta\theta \sigma(Q')+ (1-\vartheta\beta)\theta
\sum_{Q''\in\dd_Q^{k_1}\setminus\{Q'\}}\sigma(Q'')\,\leq\, \sum_{Q''\in\dd_Q^{k_1}\setminus\{Q'\}}\sigma(F_Q\cap Q'')\,.
\end{equation}
Note that we have the dyadic doubling estimate
$$\sum_{Q''\in\dd_Q^{k_1}\setminus\{Q'\}}\sigma(Q'') \leq \sigma(Q) \leq M_1 \sigma(Q')\,,$$
where $M_1=M_1(k_1,n, ADR)$.  Combining this estimate with \eqref{eq3.60}, we obtain
$$
\left[(1-\beta)\frac{\vartheta}{M_1} + (1-\vartheta\beta)\right] \theta\sum_{Q''\in\dd_Q^{k_1}
\setminus\{Q'\}}\sigma(Q'')
\leq \sum_{Q''\in\dd_Q^{k_1}\setminus\{Q'\}}\sigma(F_Q\cap Q'')\,.$$
We now choose $\beta\leq 1/(M_1+1)$, so that $(1-\beta)/M_1 \geq \beta$, and therefore the
expression in square brackets is at least 1.  Consequently, by pigeon-holing, there exists
a particular $Q''_0\in \dd_Q^{k_1}\setminus\{Q'\}$ such that 
\begin{equation}\label{eq3.62}
\theta\sigma(Q''_0) \leq \sigma(F_Q \cap Q_0'')\,.
\end{equation}
By Lemma \ref{lemma:VQ}, we can find $V_{Q''_0}$ such that
$F_Q\cap Q_0'' = F_{Q_0''}$,  where the latter is defined as in \eqref{defi-FQ}, with
$Q_0''$ in place of $Q$.  By assumption, $H[M_0,\theta]$ holds, so combining
 \eqref{eq3.62} with the fact that \eqref{eq3.12} holds with
$a=M_0$ for every $Q\in\dd(\pom)$, we find that there exists a subset 
$V^*_{Q''_0}\subset V_{Q''_0}$, along with a family of chord-arc domains 
$\{\Omega_{Q_0''}^i\}_i$ enjoying all the appropriate properties relative to $Q_0''$.
Using that $\ell(Q_0'')\approx_{k_1} \ell(Q)$, we may now proceed
exactly as in Case 1a above, and also Case 1 in Step 2,
to construct $V_Q^*$ and $\{\Omega_Q^i\}_i$ such that the conclusion of 
$H[a+b, (1-\vartheta\beta)\theta]$ holds in the present case also.

\smallskip

\noindent{\bf Case 2}:  $\mut(\dd_{Q'}) > a\sigma(Q')$ for every $Q'\in\dd_{Q}^{k_1}$.

In this case, we apply Lemma \ref{lemma:Corona}
to obtain a pairwise disjoint family $\F=\{Q_j\}\subset \dd_Q$ such 
that \eqref{Corona-sawtooth} and \eqref{Corona-bad-cubes} hold.
In particular, by our choice of $b=1/(2C)$,
$\|\mut_\F\|_{\C(Q)}\leq 1/2$. 

Recall that 
$F_Q$ is defined in \eqref{defi-FQ}, and satisfies \eqref{eq3.56}.

We define $F_0  =Q\setminus (\bigcup_\F Q_j)$ as in \eqref{eq3.27aa},
and $\F_{good}:= \F\setminus \F_{bad}$ as in \eqref{eq3.27a}.
Let
$G_0:= \bigcup_{\F_{good}}Q_j$.  Then as above (see \eqref{eq3.25}),
\begin{equation}\label{eq3.63}
\sigma(F_0 \cup G_0 )\, \geq \,\rho\sigma(Q)\,,
\end{equation}
where again $\rho=\rho(M_0,b)\in (0,1)$ is defined as in \eqref{eq3.42aa}.  Just 
as in Case 2 for $\theta=1$ in 
Subsection \ref{ss3.2}, we have that
\begin{equation}
\ell(Q_j)\le 2^{-k_1}\,\ell(Q),
\quad
\forall\,Q_j\in\F_{good},
\qquad
\mbox{and}
\qquad
\F\subset\dd_Q\setminus\{Q\}
\label{eq:wefwerfr}
\end{equation}
(see \eqref{eq:erwgte}).
Hence,  the conclusions of Claim \ref{claim3.16} hold.

We first observe that if
$\sigma(F_0) \geq \eps \sigma(Q)$, for some $\eps>0$ to be chosen
(depending on allowable parameters), then 
the desired conclusion holds.  Indeed, in this case, we may proceed exactly as in
the analogous scenario in Case 2a of Step 2:  the promised chord-arc domain is again
simply one of $\Omega_S^\pm$, since at least one of these contains a point in $V_Q$ 
and hence in particular is a subdomain of $\Omega$.
The constant $c_{a+b}$ in our conclusion will depend on $\eps$, but in the end this will be harmless,
since $\eps$ will be chosen to depend only on allowable parameters.

We may therefore suppose that 
\begin{equation}\label{eq3.65}
\sigma(F_0) < \eps \sigma(Q)\,.
\end{equation}
Next, we refine the decomposition $\F = \F_{good}\cup \F_{bad}$. 
With $\rho$ as in \eqref{eq3.42aa} and \eqref{eq3.63}, we choose $\beta< \rho/4$.
Set
$$\F^{(1)}_{good}:= \left\{Q_j\in\F_{good}:\, 
\sigma(F_Q\cap Q_j)\geq \big(1-4\vartheta\beta\rho^{-1}\big)\theta\sigma(Q_j)\right\}\,,$$
and define $\F^{(2)}_{good}:= \F_{good}\setminus \F_{good}^{(1)}$.
Let
$$\F^{(1)}_{bad}:= \left\{Q_j\in\F_{bad}:\,
\sigma(F_Q\cap Q_j)\geq \theta\sigma(Q_j)\right\}\,,$$
and define $\F^{(2)}_{bad}:= \F_{bad}\setminus \F_{bad}^{(1)}$.

We split the remaining part of Case 2 into two subcases.  The first of these will be easy,
based on our previous arguments.

\smallskip

\noindent{\bf Case 2a}:  There is $Q_j\in \F^{(1)}_{bad}$ such that $\ell(Q_j)>2^{-k_1}\,\ell(Q)$.

By definition of $\F_{bad}^{(1)}$, $\sigma(F_Q\cap Q_j)\geq \theta\sigma(Q_j)$. By 
pigeon-holing, $Q_j$ has a descendant
$Q'$ with $\ell(Q') = 2^{-k_1}\ell(Q)$, such that $\sigma(F_Q\cap Q')\geq \theta\sigma(Q')$. 
We may then  apply $H[M_0,\theta]$ in $Q'$, and proceed exactly as we did in Case 1b above with
the cube $Q_0''$, which enjoyed precisely the same properties as does our current $Q'$.
Thus, we draw the desired conclusion in the present case.

\smallskip

The main case is the following.

\noindent{\bf Case 2b}:  Every $Q_j\in \F^{(1)}_{bad}$ satisfies $\ell(Q_j)\le 2^{-k_1}\,\ell(Q)$.

Observe that by definition,
\begin{equation}\label{eq3.67}
\sigma(F_Q \cap Q_j) \leq \big(1-4\vartheta\beta\rho^{-1}\big)\theta \sigma(Q_j)\,,\qquad \forall \, Q_j \in \F^{(2)}_{good}\,,
\end{equation}
and also
\begin{equation}\label{eq3.68a}
\sigma(F_Q \cap Q_j) \leq \theta \sigma(Q_j)\,,\qquad \forall \, Q_j \in \F^{(2)}_{bad}\,,
\end{equation}

Set $\F_*:=\F\setminus \F_{good}^{(2)}$.  For future reference, we shall derive a certain ampleness estimate for
the cubes in $\F_*$.

By \eqref{eq3.56},
\begin{multline}\label{eq3.68}
(1-\vartheta\beta)\theta \sigma(Q)
\leq \, 
\sigma(F_Q)\, \leq \, \sigma(F_0) + \sum_{\F_*}\sigma(Q_j) + \sum_{\F_{good}^{(2)}}\sigma(F_Q\cap Q_j)\\[4pt]
\leq\, \eps\sigma(Q) + \sum_{\F_*}\sigma(Q_j) +
\left(1-4\vartheta\beta\rho^{-1}\right)\theta \sigma(Q)
\,,
\end{multline}
where in the last step have used \eqref{eq3.65} and \eqref{eq3.67}. Observe that
\begin{equation}\label{eq3.70a}
(1-\vartheta\beta)\theta = \left(4\rho^{-1}-1\right)\vartheta\beta\theta +\left(1-4\vartheta\beta\rho^{-1}\right)\theta \,.
\end{equation}
Using \eqref{eq3.68} and \eqref{eq3.70a}, for $\eps \ll\big(4\rho^{-1}-1\big)\vartheta\beta\theta$,
we obtain
$$
2^{-1}\left(4\rho^{-1}-1\right)\vartheta\beta\theta\sigma(Q)\le \sum_{\F_*}\sigma(Q_j)
$$
and thus
\begin{equation}\label{eq3.69}
\sigma(Q) \leq \,C(\vartheta,\rho,\beta,\theta)  \sum_{\F_*}\sigma(Q_j)\,.
\end{equation}


We now make the following claim.
\begin{claim}\label{claim3.70} For $\eps$ chosen sufficiently small,
$$\max\left(\sum_{\F_{good}^{(1)}}\sigma(Q_j)\,,\sum_{\F_{bad}^{(1)}}\sigma(Q_j)\right)\geq \eps\sigma(Q)\,.$$
\end{claim}
\begin{proof}[Proof of Claim \ref{claim3.70}]
If $\sum_{\F_{good}^{(1)}}\sigma(Q_j)\geq \eps\sigma(Q)$, then we are done.
Therefore, suppose that
\begin{equation}\label{eq3.71}
\sum_{\F_{good}^{(1)}}\sigma(Q_j)< \eps\sigma(Q)\,.
\end{equation}
We have made the decomposition 
\begin{equation}\label{eq3.74aa}
\F = \F_{good}^{(1)}\cup \F_{good}^{(2)} \cup \F^{(1)}_{bad} \cup\F_{bad}^{(2)}. 
\end{equation} 
Consequently
$$
\sigma(F_Q) \,\le\, 
  \sum_{\F_{good}^{(2)}}\sigma(F_Q\cap Q_j)
+  \sum_{\F_{bad}}\sigma(F_Q\cap Q_j)
+
O\left(\eps\sigma(Q)\right)
\,, 
$$
where we have used \eqref{eq3.65}, and 
\eqref{eq3.71} to estimate the contributions of $F_0$,  and of
$\F_{good}^{(1)}$, respectively.  This,  \eqref{eq3.56}, \eqref{eq3.67}, and \eqref{eq3.68a} yield
\begin{multline*}
(1-\vartheta\beta)\theta \left( \sum_{\F_{good}^{(2)}}\sigma(Q_j)
+  \sum_{\F_{bad}^{(2)}}\sigma(Q_j) \right)
\leq \, (1-\vartheta\beta)\theta \sigma(Q)\leq\,\sigma(F_Q) 
\\[4pt]
\leq\,
\left(1-4\vartheta\beta\rho^{-1}\right)\theta  \sum_{\F_{good}^{(2)}}\sigma(Q_j)
+ \sum_{\F_{bad}^{(1)}}\sigma(Q_j)
+\theta\sum_{\F_{bad}^{(2)}}\sigma(Q_j) 
+ O\left(\eps\sigma(Q)\right)
\,.
\end{multline*}
In turn, applying  \eqref{eq3.70a} in the latter estimate, and rearranging terms,
we obtain
\begin{equation}\label{eq3.74}
(4\rho^{-1} -1)\vartheta\beta\theta  \sum_{\F_{good}^{(2)}}\sigma(Q_j)
- \vartheta\beta\theta \sum_{\F_{bad}^{(2)}}\sigma(Q_j) 
\leq\,
 \sum_{\F_{bad}^{(1)}}\sigma(Q_j)
+ O\left(\eps\sigma(Q)\right)
\,.
\end{equation}
Recalling that $G_0=\cup_{\F_{good}}Q_j$, and that
$\F_{good}= \F^{(1)}_{good} \cup \F^{(2)}_{good}$,
we further note that by \eqref{eq3.63}, choosing $\eps \ll \rho$, and using
\eqref{eq3.65} and 
\eqref{eq3.71}, we find in particular that
\begin{equation}\label{eq3.75} 
\sum_{\F_{good}^{(2)}}\sigma(Q_j) \,\geq\, \frac{\rho}{2} \sigma(Q). 
\end{equation}
Applying  
\eqref{eq3.75} and the trivial estimate
$ \sum_{\F_{bad}^{(2)}}\sigma(Q_j)\leq\sigma(Q)$ in \eqref{eq3.74}, we then have 
\begin{multline*}
\vartheta\beta\theta\left[1-\frac{\rho}{2}\right]\,\sigma(Q)
=
\left[ \big(4\rho^{-1}-1\big)\vartheta\beta\theta\frac{\rho}{2} -\vartheta\beta\theta \right]\,\sigma(Q)
\\
\leq\, 
\big(4\rho^{-1}-1\big)\vartheta\beta\theta\sum_{\F_{good}^{(2)}}\sigma(Q_j) 
-\vartheta\beta\theta \sum_{\F_{nbad}^{(2)}}\sigma(Q_j) 
\le
\sum_{\F_{bad}^{(1)}}\sigma(Q_j) 
+ O\left(\eps\sigma(Q)\right)\,.
\end{multline*}
Since $\rho<1$, we conclude, 
for $\eps\le (4C)^{-1}\vartheta\beta\theta$, that 
$$
\frac14\vartheta\beta\theta\, \sigma(Q)
\,\le\,
\sum_{\F_{bad}^{(1)}}\sigma(Q_j)\,,
$$
and
Claim \ref{claim3.70}  follows.
\end{proof}

With Claim \ref{claim3.70} in hand, let us return to the proof of Case 2b of Claim \ref{claim3.52}.

We begin by noting that by definition of $\F_{bad}^{(1)}$, and Lemma \ref{lemma:VQ}, we can apply 
$H[M_0,\theta]$ to any $Q_j \in \F_{bad}^{(1)}$, hence for each
such $Q_j$ there is a family of chord-arc domains $\{\Omega^i_{Q_j}\}_i$ satisfying the desired properties.

Now consider $Q_j\in\F^{(1)}_{good}$.  Since $\F^{(1)}_{good}\subset \F_{good}$, 
by pigeon-holing $Q_j$ has a dyadic
child $Q_j'$  satisfying 
\begin{equation}\label{eq3.60*}
\mut(\dd_{Q'_j}) \leq a\sigma(Q'_j)\,,
\end{equation}
(there may be more than one such child, but we just pick one).
Our immediate goal is to find a child $Q_j''$ of $Q_j$, which may or may not equal $Q_j'$,
for which we may construct a family of chord-arc domains  $\{\Omega^i_{Q''_j}\}_i$ satisfying the desired properties.
To this end, we assume first that $Q_j'$ satisfies 
\begin{equation}\label{eq3.61}
\sigma(F_Q\cap Q_j') \ge  (1-\vartheta)\theta \sigma (Q_j')\,.
\end{equation}
In this case, we set $Q_j'':=Q_j'$, and using Lemma \ref{lemma:VQ},
by the induction hypothesis
$H[a,(1-\vartheta)\theta]$, we obtain the desired family of chord-arc domains.

We therefore consider the case 
\begin{equation}\label{eq3.80}
\sigma(F_Q\cap Q_j') < (1-\vartheta)\theta \sigma (Q_j')\,.
\end{equation}
In this case, we shall select $Q_j''\neq Q_j'$.
Recall that we use the notation $Q''\lhd Q$ to mean that $Q''$ is a dyadic child of $Q$.
Set $$\F_j'':=\left\{Q_j''\lhd Q_j: \, Q_j''\neq Q_j'\right\}\,.$$
Note that we have the dyadic doubling estimate
\begin{equation}
\sum_{Q_j''\in \F_j''}\sigma(Q_j'') \leq \sigma(Q_j) \leq M_1 \sigma(Q_j')\,,
\label{eq:dd-est}
\end{equation}
where $M_1=M_1(n, ADR)$. 
We also note that
\begin{equation}\label{eq3.81}
\big(1-4\vartheta\beta\rho^{-1}\big) \theta= \big(1-4\beta\rho^{-1}\big)\vartheta\theta + (1-\vartheta)\theta\,.
\end{equation}
By definition of $\F_{good}^{(1)}$, 
$$\big(1-4\vartheta\beta\rho^{-1}\big) \theta \sigma(Q_j)\, \leq \,\sigma(F_Q\cap Q_j) \,=\, \sigma(F_Q\cap Q_j') + 
\sum_{Q_j''\in \F_j''}\sigma(F_Q\cap Q_j'')\,.$$
By \eqref{eq3.80}, it follows that
\begin{multline*}\big(1-4\vartheta\beta\rho^{-1}\big)\theta \sigma(Q_j')+ \big(1-4\vartheta\beta\rho^{-1}\big)\theta \!\!
\sum_{Q_j''\in\F_j''}\sigma(Q_j'')=\big(1-4\vartheta\beta\rho^{-1}\big)\theta \sigma(Q_j)\\[4pt] \leq \, (1-\vartheta)\theta \sigma(Q_j') + 
\sum_{Q_j''\in\F_j''}\sigma(F_Q\cap Q_j'')\,.
\end{multline*}
In turn, using \eqref{eq3.81}, we obtain
$$\big(1-4\beta\rho^{-1}\big)\vartheta\theta\sigma(Q_j') + \big(1-4\vartheta\beta\rho^{-1}\big)\theta \!\!
\sum_{Q_j''\in\F_j''}\sigma(Q_j'') \,\leq\, \sum_{Q_j''\in\F_j''}\sigma(F_Q\cap Q_j'')\,.$$
By the dyadic doubling estimate \eqref{eq:dd-est}, this leads to
$$
\left[\big(1-4\beta\rho^{-1}\big)\vartheta M_1^{-1} + \big(1-4\vartheta\beta\rho^{-1}\big) \right]\theta
\sum_{Q_j''\in\F_j''}\sigma(Q_j'') \,\leq\, \sum_{Q_j''\in\F_j''}\sigma(F_Q\cap Q_j'')\,.
$$
Choosing $\beta \leq \rho/ (4 (M_1+1))$, we find that the
expression in square brackets is at least 1, and therefore, by pigeon holing,
we can pick $Q_j''\in\F_j''$ satisfying
\begin{equation}\label{eq3.82}
\sigma(F_Q\cap Q_j'')\geq \theta \sigma(Q_j'')\,.
\end{equation}  

Hence, using Lemma \ref{lemma:VQ}, we see that
the induction hypothesis 
$H[M_0,\theta]$ holds for $Q''_j\in \F_j''$, and once again we obtain the desired family of chord-arc domains.

Recall that we have constructed our packing measure $\mut$ in such a way that
each $Q_j\in \F$, as well as all of its children, along with the cubes in $\dd_\F\cap\dd_Q$,
belong to the same stopping time regime $\sbf$; see Claim \ref{claim3.16}.  This means in particular
that for each such $Q_j$, the Whitney region $U_{Q_j}$ has exactly two components $U_{Q_j}^\pm\subset
\Omega_\sbf^\pm$, and the analogous statement is true for each child of $Q_j$.   This fact has the following 
consequences:

\begin{remark}\label{remark3.66}
For each $Q_j \in \F_{bad}^{(1)}$, and for the selected child $Q_j''$ of each $Q_j\in \F_{good}^{(1)}$,
the conclusion of the induction hypothesis produces at most two 
chord-arc domains $\Omega^\pm_{Q_j} \supset U_{Q_j}^\pm$ (resp. $\Omega_{Q_j''}^\pm
\supset U_{Q_j''}^\pm$), which we enumerate as $\Omega^i_{Q_j}$ (resp. $\Omega_{Q_j''}^i)$,
$i=1,2$, with $i=1$ corresponding 
 ``+", and  $i=2$ corresponding to ``-", respectively.   
\end{remark}
 
\begin{remark}\label{remark3.67}
For each $Q_j\in \F_{good}^{(1)}$, the connected component $U_{Q_j}^\pm$ overlaps with
the corresponding component $U^\pm_{Q_j''}$ for its child,  so we may augment
$\Omega_{Q_j''}^i$ by adjoining to it the appropriate component $U_{Q_j}^\pm$, to form
a chord arc domain $$\Omega^i_{Q_j} := \Omega_{Q_j''}^i \cup U_{Q_j}^i\,.$$  
\end{remark}


By the induction hypothesis, for each $Q_j \in \F_{bad}^{(1)} \cup \F_{good}^{(1)}$
(and by ADR, in the case of $\F_{good}^{(1)}$),
the chord-arc domains $\Omega^i_{Q_j}$ that we have constructed satisfy
$$\sum_i \sigma (Q_j\cap \pom_{Q_j}^i) \gtrsim \sigma(Q_j)\,,$$
where the sum has either one or two terms,  and where 
the implicit constant depends either on $M_0$ and $\theta$, or on $a$ and $(1-\vartheta)\theta$,
depending on which part of the induction hypothesis we have used.
In particular, for each such $Q_j$, there is at least one choice of index $i$  such that
$\Omega^i_{Q_j}=:\Omega_{Q_j}$ satisfies
\begin{equation}\label{eq3.69a}
\sigma (Q_j\cap \pom_{Q_j}) \gtrsim \sigma(Q_j)
\end{equation}
(if the latter is true for both choices $i=1,2$, we arbitrarily choose $i=1$, 
which we recall corresponds to ``+").
Combining the latter bound with Claim \ref{claim3.70}, and recalling that $\eps$ has now 
been fixed depending only on allowable parameters,
we see that
\begin{equation*}
\sum_{Q_j\, \in \, \F_{bad}^{(1)} \,\cup\, \F_{good}^{(1)}}\sigma (Q_j\cap \pom_{Q_j}) \gtrsim \sigma(Q)
\end{equation*}
 For $Q_j \in  \F_{bad}^{(1)} \,\cup\, \F_{good}^{(1)}$, as above set 
$B^*_{Q_j}:= B(x_{Q_j}, K \ell(Q_j))$. 
 By a covering lemma argument, 
we may extract a subfamily 
$\F^*\subset \F_{bad}^{(1)} \,\cup\, \F_{good}^{(1)}$ 
such that $ \{\kappa B^*_{Q_j}\}_{Q_j\in \F^*}$ is 
pairwise disjoint, where again $\kappa\gg K^4$ is a large dilation factor, and such that 
\begin{equation}\label{eq3.69*}
\sum_{Q_j\, \in \, \F^*}\sigma (Q_j\cap \pom_{Q_j}) \gtrsim_\kappa \sigma(Q)
\end{equation}

Let us now build (at most two) chord-arc domains $\Omega^i_Q$ satisfying the desired properties.


Recall that for each $Q_j \in \F^*$, we defined the corresponding chord-arc
domain $\Omega_{Q_j}:= \Omega_{Q_j}^i$, where the choice of index $i$ (if there was a choice),
was made so that \eqref{eq3.69a} holds.  We then assign each $Q_j\in \F^*$ either to
$\F^*_+$ or to $\F^*_-$, depending on whether we chose $\Omega_{Q_j}$ satisfying
\eqref{eq3.69a} to be $\Omega_{Q_j}^1=\Omega_{Q_j}^+$, or $\Omega_{Q_j}^2
=\Omega_{Q_j}^-$.  We note that at least one of the sub-collections $\F^*_\pm$ 
is non-empty, since
for each $j$, there was at least one choice of index $i$ such that
\eqref{eq3.69a} holds with $\Omega_{Q_j}:= \Omega_{Q_j}^i$.  Moreover, the 
two collections are disjoint, since we have arbitrarily designated $\Omega_{Q_j} = \Omega_{Q_j}^1$
(corresponding to ``+") in the case that there were two choices for a particular $Q_j$.

We further note that if $Q_j \in \F^*_\pm$, then $\Omega_{Q_j} = \Omega_{Q_j}^\pm\supset
U_{Q_j}^\pm$. 

We are now in position to apply Lemma \ref{lemmaCAD}.  Set
\[
	\sbf^*=\left\{Q'\in \dd_Q: Q_j\subset Q' \text{ for some }Q_j\in\F^*\right\}\,,
	\] 
which is a semi-coherent subregime of $\sbf$, with maximal cube $Q$.	
Without loss of generality, we may suppose that $\F^*_+$ is non-empty, and we then define
$$  \Omega_Q^+:= \Omega_{\sbf^*}^+ \bigcup
 \left(\bigcup_{Q_j \,\in\, \F^*_+}\Omega_{Q_j}\right)\,,$$
 and similarly with ``+" replaced by ``-",  provided that $\F^*_-$ is also non-empty.
 Observe that by the induction hypothesis, 
 and our construction (see Remarks \ref{remark3.66}-\ref{remark3.67} and Lemma \ref{lemma2.7}),
 for an appropriate choice of $\pm$,
 $U_{Q_j}^\pm\subset \Omega_{Q_j} \subset B_{Q_j}^*$, and since $\ell(Q_j) \leq 2^{-k_1} \ell(Q)$,
by \eqref{eq3.69*} and Lemma  \ref{lemmaCAD}, 
each (non-empty) choice defines a chord-arc domain with the requisite properties.
 \end{proof}

\appendix

\section{Some counter-examples} \label{appa}

We shall discuss some counter-examples which show that
our background hypotheses  in Theorem \ref{tmain} (namely,
ADR and interior Corkscrew condition)
are in some sense in the nature of best possible.  In the first two examples, 
$\Omega$ is a domain satisfying an interior Corkscrew condition, such that $\pom$
satisfies exactly one
(but not both) of the upper or the lower ADR bounds, and for which harmonic measure $\hm$ 
fails to be weak-$A_\infty$ with respect to surface measure $\sigma$ on $\pom$.
In this setting, in which full ADR
fails, there is no established notion of uniform rectifiability, but in each case, the domain
will enjoy some substitute property which would imply uniform rectifiability of the boundary in the 
presence of full ADR.  

In the last example, we construct an open set $\Omega$ with ADR boundary, and for which
$\hm\in$ weak-$A_\infty$ with respect to surface measure, but for which the interior
Corkscrew condition fails, and $\pom$ is not uniformly rectifiable.

\smallskip

\noindent{\it Failure of the upper ADR bound.}
In \cite{AMT1}, the authors construct an example of a Reifenberg flat domain $\Omega \subset \ree$ 
for which surface measure $\sigma= H^n\lfloor_{\pom}$  is locally finite on $\pom$, but for 
which the
upper ADR bound
\begin{equation}\label{a1}
\sigma(\Delta(x,r) \leq C r^n
\end{equation}
fails, and for which harmonic measure $\hm$ is not absolutely continuous with respect to $\sigma$.  Note that
the hypothesis of Reifenberg flatness implies in particular that $\Omega$ and 
$\Omega_{ext}:= \ree\setminus \overline{\Omega}$ are both NTA domains, hence both enjoy the
Corkscrew condition, so by the relative isoperimetric inequality, the lower ADR bound
\begin{equation}\label{a2}
\sigma(\Delta(x,r) \geq c r^n
\end{equation} holds.  
Thus, it is the failure of \eqref{a1} which causes the failure of absolute continuity: in the presence of
\eqref{a1}, the results of \cite{DJe} apply, and one has that $\hm \in A_\infty(\sigma)$, and that $\pom$ satisfies
a ``big pieces of Lipschitz graphs" condition (see \cite{DJe} for a precise statement), and hence is
uniformly rectifiable.  We note that by a result of Badger \cite{Ba}, a version of the Lipschitz approximation
result of \cite{DJe} still holds for NTA domains with locally finite surface measure, even in the absence of the upper
ADR condition.

\smallskip

\noindent{\it Failure of the lower ADR bound.} In \cite[Example 5.5]{ABoHM}, the authors give an 
example of a domain satisfying
the interior Corkscrew condition, whose boundary is rectifiable  (indeed, 
it is contained in a countable
union of hyperplanes), and 
satisfies the upper ADR condition \eqref{a1}, but 
not the lower ADR condition \eqref{a2},  but for which
surface measure $\sigma$ fails to be absolutely continuous with respect to harmonic measure, and in fact,
for which the non-degeneracy condition
\begin{equation}\label{a3}
A\subset \Delta_X:= B(X,10\delta(X)) \cap\pom,\quad \sigma(A)\geq (1-\eta) \sigma(\Delta_X)  \, \, \implies \,\,  \hm^X(A) \geq c\,,
\end{equation}
fails to hold uniformly for $X\in\Omega$, for any fixed positive $\eta$ and $c$, and therefore $\hm$ cannot be weak-$A_\infty$ with respect to
$\sigma$.  We note that in the presence of the full ADR condition,
if $\pom$ were contained in a countable union of hyperplanes (as it is in the example), 
then in particular it would satisfy
the ``BAUP" condition of \cite{DS2}, and thus would be uniformly rectifiable \cite[Theorem I.2.18, p. 36]{DS2}. 

\smallskip

\noindent{\it Failure of the interior Corkscrew condition}. The example is based on the construction of Garnett's 4-corners
Cantor set $\mathcal{C}\subset\re^2$
(see, e.g., \cite[Chapter 1]{DS2}).   Let $I_0$ be a unit square
positioned with lower left corner
at the origin in the plane, and in general for each $k = 0, 1,2,...$, we let $I_k$ be the unit square positioned with lower left corner at the point $(2k,0)$ on the $x$-axis.
Set $\Omega_0:= I_0$.
Let $\Omega_1$ be the first stage of the 4-corners construction, i.e., a union of four squares
of side length 1/4, positioned in the corners of the unit square $I_1$, and similarly, for each $k$, let
$\Omega_k$ be the $k$-th stage of the 4-corners construction, positioned inside $I_k$.  Note that
$\dist(\Omega_k, \Omega_{k+1}) = 1$ for every $k$.  Set $\Omega := \cup_k \Omega_k$. It is easy to check that 
$\pom$ is ADR, and that the non-degeneracy condition
\eqref{a3} holds in $\Omega$ for some uniform positive $\eta$ and $c$, and thus by the criterion of \cite{BL},
$\hm\in$ weak-$A_\infty(\sigma)$.  On the other hand, the interior Corkscrew condition clearly fails to hold
in $\Omega$ (it holds only for decreasingly small scales as $k$ increases), and certainly $\pom$ cannot
be uniformly rectifiable:  indeed, if it were, then $\pom_k$ would be UR, with uniform constants, for each $k$,
and this would imply that $\mathcal{C}$ itself was UR, whereas in fact, as is well known, it is totally non-rectifiable. One can produce a similar set in 3 dimensions by simply taking the cylinder $\Omega'=\Omega\times[0,1]$. Details are left to the interested reader.

\end{document}